\documentclass[a4paper]{amsart} 
\usepackage{amssymb,times, amscd,amsmath,amsthm, xypic}
\usepackage[all]{xy}
\usepackage{geometry}
\usepackage{mathrsfs}
\author{Shoji Yokura}

\address
{Graduate School of Science and Engineering,
Kagoshima University, 21-35 Korimoto 1-chome, Kagoshima 890-0065, Japan}
\email {yokura@sci.kagoshima-u.ac.jp}
\title [Cobordism bicycles]
{Cobordism bicycles of vector bundles}

\thanks {
\quad \emph{keywords} : (co)bordism, algebraic cobordism, algebraic cobordism of bundles, correspondence }

\numberwithin{equation}{section}
\newtheorem{thm}[equation]{Theorem}
\newtheorem{pro}[equation]{Proposition}

\newtheorem{cor}[equation]{Corollary}
\newtheorem{lem}[equation]{Lemma}
\theoremstyle{definition}

\newtheorem{defn}[equation]{Definition}

\newtheorem{rem}[equation]{Remark}

\def\alp{\alpha}
\def\be{\beta}
\def\jeden{1\hskip-3.5pt1}

\def\bigstar{\mathbf{\star}}

\def\Cal{\mathscr}
\def\ga{\gamma}

\def \bB{\mathbb B}

\def \bOB{\mathbb {OB}}

\def \bOM{\mathbb {OM}}

\def\bZ{\mathbb Z}

\def\op{\operatorname}

\begin{document} 
 
\begin{abstract}  
      The main ingredient of the algebraic cobordism 
      of M. Levine and F. Morel is a cobordism cycle of the form $(M \xrightarrow {h}  X; L_1, \cdots, L_r)$ with 
 a proper map $h$ from a smooth variety $M$ and line bundles $L_i$'s over $M$. 
In this paper we consider a \emph{cobordism bicycle} of a finite set of line bundles $(X \xleftarrow p V \xrightarrow s Y; L_1, \cdots, L_r)$ with a proper map 
$p$ and a smooth map 
$s$ and line bundles $L_i$'s  over $V$. We will show that the Grothendieck group $\mathscr Z^*(X, Y)$ of the abelian monoid of the isomorphism classes of cobordism bicycles of finite sets of line bundles satisfies properties similar to those of Fulton-MacPherson's bivariant theory and also that $\mathscr Z^*(X, Y)$ is a \emph{universal} one among such abelian groups, i.e., for any abelian group $\mathscr B^*(X, Y)$ satisfying the same properties there exists a unique Grothendieck transformation $\gamma: \mathscr Z^*(X,Y) \to \mathscr B^*(X,Y)$ preserving the unit.
\end{abstract} 

\maketitle


\section{Introduction}\label{intro} 
V. Voevodsky first introduced algebraic cobordism or higher algebraic cobordism $\op{MGL}^{*,*}(X)$ in the context of motivic homotopy theory and used it in his proof of the Milnor conjecture \cite{Vo1, Vo2, Vo3}. 
Later, in an attempt to understand $\op{MGL}^{*,*}(X)$ better,   
M. Levine and F. Morel \cite{LM} constructed another algebraic cobordism $\Omega_*(X)$ in terms of what they call \emph{a cobordism cycle} (which is of the form $[V \xrightarrow h X; L_1, L_2, \cdots, L_r]$ with line bundles $L_i$'s over $V$ which is  smooth) and some relations on these cobordism cycles, as the universal oriented cohomology theory. To be a bit more precise, using cobordism cycles they first defined \emph{an oriented Borel--Moore functor $\mathcal Z_*$ with products}  satisfying twelve conditions (D1) - (D4) and (A1) - (A8), and then defined \emph{an oriented Borel--Moore functor with products of geometric type} by further imposing on the functor $\mathcal Z_*$ the relations $\mathcal R$ corresponding to three axioms (Dim) (dimension axiom), (Sect) (section axiom) and (FGL) (formal group law axiom), which correspond to \emph{``of geometric type''}. The functor  $\mathcal Z_*/\mathcal R$ is nothing but Levine--Morel's algebraic cobordism $\Omega_*$. In \cite{Le} M. Levine showed that there is an isomorphism $\Omega_*(X) \cong \op{MGL}^{2*,*}(X)$ for smooth $X$.

 In \cite{FM} W. Fulton and R. MacPherson have introduced \emph{bivariant theory} $\mathbb B(X \xrightarrow f Y)$  with an aim to deal with Riemann--Roch type theorems for singular spaces and to unify them. The extreme cases $\mathbb B_*(X):= \mathbb B^{-*}(X \xrightarrow {\pi_X} pt)$ becomes a covariant functor and $\mathbb B^*(X):= \mathbb B^{*}(X \xrightarrow {\op{id}_X} X)$ becomes a contravariant functor. In this sense  $\mathbb B(X \xrightarrow f Y)$ is called a \emph{bivariant} theory. In \cite{Yokura-obt} (cf. \cite{Yokura-obt2}) the author introduced \emph{an oriented bivariant theory} and \emph{a universal oriented bivariant theory} in order to construct a bivariant-theoretic version $\Omega^*(X \xrightarrow f Y)$ of Levine--Morel's algebraic cobordism so that the covariant part $\Omega^{-*}(X \xrightarrow {\pi_X} pt)$ becomes isomorphic to Levine--Morel's algebraic cobordism $\Omega_*(X)$.
 
 Our universal oriented bivariant theory $\mathbb {OM}_{sm}^{prop}(X \xrightarrow f Y)$ is defined to be the Grothendieck group of the abelian monoid of the isomorphism classes $[V \xrightarrow p X; L_1, L_2, \cdots, L_r]$ such that
\begin{enumerate}
\item $h: V \to X$ is a \emph{proper} map,
\item the composite $f \circ h: V \to Y$ is a \emph{smooth} map. 
(Note that this requirement implies that if the target $Y$ is the point $pt$, then the source $V$ has to be smooth, thus $[V \xrightarrow h X; L_1, L_2, \cdots, L_r]$ becomes a cobordism cycle in the sense of Levine--Morel.)
\end{enumerate}
Here for the monoid we consider the following addition 
\begin{eqnarray} \label{add}
\quad \quad \quad [V_1 \xrightarrow {p_1} X; L_1, \cdots, L_r]   + [V_2  \xrightarrow {p_2} X; L'_1, \cdots, L'_r] \hspace{3cm} \\
 := [V_1 \sqcup V_2 \xrightarrow {p_1 \sqcup p_2} X; L_1 \sqcup L'_1, \cdots, L_r \sqcup L'_r] \nonumber
\end{eqnarray}
where $\sqcup$ is the disjoint sum and $L_i \sqcup L'_i$  is a line bundle over $V_1 \sqcup V_2$ such that $(L_i \sqcup L'_i)|_{V_1} =L_i$ and $(L_i \sqcup L'_i)|_{V_2} =L'_i$.
In other words $\mathbb {OM}_{sm}^{prop}(X \xrightarrow f Y)$ is the free abelian group generated by the isomorphism classes $[V \xrightarrow p X; L_1, L_2, \cdots, L_r]$ modulo the additivity relation (\ref{add}).\footnote{We remark that in \cite{Yokura-obt} we consider the free abelian group generated by the isomorphism classes $[V \xrightarrow p X; L_1, L_2, \cdots, L_r]$, however the results in \cite{Yokura-obt} still hold even if we consider the Grothendieck group by modding it out by the additivity relation (\ref{add}).}

\begin{equation}\label{forget}
\xymatrix{
& \{L_1, \cdots L_r\} \ar[d]  \\
& V \ar[dl]_h \ar[dr]^{f \circ h} &\\
X \ar[rr]_f &&  Y
} \overset {\text{forget the map $f$ }}{======\Longrightarrow } 
\xymatrix{
& \{L_1, \cdots L_r\} \ar[d]  \\
& V \ar[dl]_h \ar[dr]^{f \circ h} &\\
X  &&  Y
}
\end{equation}
If we forget or ignore the given map $f$ in the left-hand-side diagram above, then we get the right-hand-side diagram above, which is 
$$\text{\emph{a correspondence $X \xleftarrow h V \xrightarrow {f \circ h}  Y$ 
with a finite set of line bundles $\{ L_1, \cdots, L_r \},$}}$$
where  $h: V \to X$ is a proper map and $f \circ h :V \to Y$ is a smooth map. Such a correspondence (or sometimes called a span or a roof) can be considered for any pair $(X, Y)$ of varieties $X$ and $Y$:
$$
\xymatrix{
& \{L_1, \cdots L_r\} \ar[d]  \\
& V \ar[dl]_p \ar[dr]^s &\\
X &&  Y, 
}  \qquad \xymatrix{
&& E \ar[d]  \\
&& V \ar[dll]_p \ar[drr]^s &&\\
X &&&&  Y
}
$$
with a proper map $p$ and a smooth map $s$. In the right-hand-side diagram $E$ is one vector bundle over $V$, not necessarily a line bundle. These two correspondences are denoted by
$(X \xleftarrow p V \xrightarrow s Y; L_1, \cdots L_r)$ and $(X \xleftarrow p V \xrightarrow s Y; E)$ respectively.

It turns out that such a correspondence has been already studied in $C^*$-algebra, in particular for Kasparov's KK-theory $KK(X, Y)$ (\cite{Ka}), \emph{another kind of bivariant theory} which has been studied by many people in operator theory. For example, 
in \cite{CS} (cf. \cite{BD} and \cite{EM3})  A. Connes and G. Skandalis consider $(X \xleftarrow b M \xrightarrow f Y; \xi)$ with $b$ a proper map, $f$ a smooth $K$-oriented map and $\xi$ a vector bundle over $M$. In \cite{BB} P. Baum and J. Block consider such a correspondence for singular spaces with a group action on it and call such a correspondence \emph{an equivariant bicycle}. So we shall call the above correspondence \emph{a cobordism bicycle of vector bundles}.
 
In our previous paper \cite{Yokura-enriched} we consider the above cobordism bicycle $(X \xleftarrow p V \xrightarrow s Y; E)$ of vector bundles as \emph{a morphism from $X$ to $Y$} and furthermore consider the enriched category of such cobordism bicycles of vector bundles. Then we extend Baum--Fulton--MacPherson's Riemann--Roch (or Todd class) transformation $\tau: G_0(-) \to H_*(-) \otimes \mathbb Q$ (see \cite{BFM}) to this enriched category. 

If we consider the above (\ref{forget}), it is quite natural or reasonable to think that there must be a connection or a relation between the above two kinds of bivariant theories, Fulton--MacPherson's bivariant theory (in topology) and Kasparov's bivariant theory (in operator theory). So, \emph{as an intermediate theory between these two theories} we consider the free abelian group $\mathscr Z^{*}(X, Y)$ generated by isomorphism classes of cobordism bicycles of vector bundles and 
finite sets of line bundles modulo the additive relation like (\ref{add}):
\begin{equation*}\label{add-vect}
[X \xleftarrow {p_1} V_1 \xrightarrow {s_1} Y; E_1] + [X \xleftarrow {p_2} V_2 \xrightarrow {s_2} Y; E_2]:= [X \xleftarrow {p_1 \sqcup p_2 } V_2 \xrightarrow {s_1 \sqcup s_2} Y; E_1 \sqcup E_2],
\end{equation*}
\begin{align*}\label{add-lines}
[X \xleftarrow {p_1} V_1 \xrightarrow {s_1} Y; L_1, \cdots, L_r] & + [X \xleftarrow {p_2} V_2 \xrightarrow {s_2} Y; L'_1, \cdots, L'_r] \\
& := [X \xleftarrow {p_1 \sqcup p_2 } V_2 \xrightarrow {s_1 \sqcup s_2} Y; L_1 \sqcup L'_1, \cdots, L_r \sqcup L'_r].
\end{align*}

For example we can show the following
\begin{thm}
For a pair $(X,Y)$ the Grothendieck group $\mathscr Z^{*}(X, Y)$ of the abelian monoid of the isomorphism classes of cobordism bicycles of finite sets of line bundles satisfies the following (similar to those of Fulton--MacPherson's bivariant theory):

\noindent
(1) it is equipped with the following three operations

\begin{enumerate}
\item (product)  \quad $\bullet: \mathscr Z^i(X, Y) \otimes \mathscr Z^j(Y,Z) \to \mathscr Z^{i+j}(X, Z)$

\item (Pushforward) 
\begin{enumerate}
\item For a \emph{proper} map $f:X \to X'$, $f_*:\mathscr Z^i(X,Y) \to \mathscr Z^i(X',Y)$.
\item For a \emph{smooth} map $g:Y \to Y'$, ${}_*g \footnote{For this unusual notation ${}_*g$ instead of $g_*$ see \S 4}: \mathscr Z^i(X,Y) \to \mathscr Z^{i+\op{dim}g}(X,Y')$.
\end{enumerate}

\item (Pullback) 
\begin{enumerate}
\item For a \emph{smooth} map $f:X' \to X$, $f^*: \mathscr Z^i(X,Y) \to  \mathscr Z^{i+\op{dim}f}(X',Y)$.
\item For a \emph{proper} map $g:Y' \to Y$, ${}^*g \footnote{For this unusual notation ${}^*g$ instead of $g^*$ see \S 4}: \mathscr Z^i(X,Y) \to  \mathscr Z^i(X,Y')$.
\end{enumerate}

\end{enumerate}

(2) the three operations satisfies the following nine properties:
\begin{enumerate}
\item[($A_1$)] Product is associative.
\item[($A_2$)] Pushforward is functorial.
\item[($A_2$)'] Proper pushforward and smooth pushforward commute.
\item[($A_3$)] Pullback is functorial.
\item[($A_3$)'] Proper pullback and smooth pullback commute.
\item[($A_{12}$)] Product and pushforward commute. 
\item[($A_{13}$)] Product and pullback commute.
\item[($A_{23}$)] Pushforward and pullback commute.
\item[($A_{123}$)] Projection formula.
\end{enumerate}
 
\noindent
(3) $\mathscr Z^*$ has units, i.e., there is an element $1_X \in \mathscr Z^0(X,X)$ such that $1_X \bullet \alp = \alp$ for any element $\alp \in \mathscr Z^*(X, Y)$ and $\be \bullet 1_X = \be$ for any element $\be \in \mathscr Z^*(Y, X)$.

\noindent
(4) $\mathscr Z^*$ satisfies PPPU (Pushforward-Product Property for units) and PPU (Pullback Property for units). (For the details of these properties, see Lemma \ref{pppu} and Lemma \ref{ppu}.)

\noindent
(5) $\mathscr Z^*$ is equipped with the Chern class operators: for a line bundle $L$ over $X$ and a line bundle $M$ over $Y$
$$c_1(L) \bullet: \mathscr Z^i(X,Y) \to \mathscr Z^{i+1}(X,Y), \bullet c_1(M): \mathscr Z^i(X,Y) \to \mathscr Z^{i+1}(X,Y)$$
which satisfy the properties listed in Lemma \ref{ch-op}.

\end{thm}
We do not have a reasonable name for this naive theory $\mathscr Z^{*}(X, Y)$ satisfying those properties above, so in this paper we call it \emph{a ``bi-variant" theory}.
\begin{defn}
Let $\Cal   B, \Cal   B'$ be two bi-variant theories on a category $\Cal V$. A {\it Grothen- dieck transformation} $\ga : \Cal   B \to \Cal   B'$
is a collection of homomorphisms
$\Cal   B(X, Y) \to \Cal   B'(X,Y)$
for a pair $(X,Y)$ in the category $\Cal V$, which preserves the above three basic operations and the Chern class operator: 
\begin{enumerate}
\item $\ga (\alp \bullet_{\Cal   B} \be) = \ga (\alp) \bullet _{\Cal   B'} \ga (\be)$, 
\item $\ga(f_{*}\alp) = f_*\ga (\alp)$ and $\ga(\alp \, {}_*g) = \ga (\alp) \, {}_*g$ ,
\item $\ga (g^* \alp) = g^* \ga (\alp)$ and $\ga (\alp \, {}^*f) = \ga (\alp)\, {}^*f$
\item $\ga(c_1(L) \bullet \alp) = c_1(L) \bullet \ga(\alp)$ and $\ga(\alp \bullet c_1(M)) = \ga(\alp) \bullet c_1(M).$ 
\end{enumerate}
\end{defn}

We show the following theorem.
\begin{thm}
The above $\mathscr Z^*(-,-)$ is the universal one among bi-variant theories in the sense that given any bi-variant theory $\Cal  B^*(-,-)$, there exists a unique Grothendieck transformation
$$\gamma_{\Cal  B}: \Cal  Z^*(-,-) \to \Cal  B^*(-,-)$$
such that $\gamma_{\Cal  B}(\jeden_V) = 1_V \in \Cal B(V,V)$ for any variety $V$.
\end{thm}

\begin{rem} In \cite{An} T. Annala has succeeded in constructing what he calls \emph{the bivariant derived algebraic cobrodism $\Omega^*(X \xrightarrow f Y)$}, a bivariant theoretic analogue of Levine--Morel's algebraic cobordism $\Omega_*(X)$ (which the author has been trying to aim at)
, using the construction of Lowrey-Sch\"urg's derived algebraic cobordism $d\Omega_*(X)$ \cite{LS} in \emph{derived algebraic geometry} and the author's construction of the universal bivariant theory. Roughly speaking, in \cite{An} Annala considers the bivariant theory $\mathbb {OM}_{qusm}^{prop}(X \xrightarrow f Y)$ for the category of derived algebraic schemes in derived algebraic geometry, where $qusm$ refers to \emph{quai-smooth morphisms},  and furthermore imposes some relations $\Cal R(X \xrightarrow f Y)$ on $\mathbb {OM}_{qusm}^{prop}(X \xrightarrow f Y)$ to obtain its quotient $\mathbb {OM}_{qusm}^{pro}(X \xrightarrow f Y)/\Cal R(X \xrightarrow f Y)$
, which is the bivariant derived algebraic cobordism $\Omega^*(X \xrightarrow f Y)$.

The ``forget" map defined in (\ref{forget}) gives rise to the following canonical homomorphism 
$$\frak f: \mathbb {OM}^{prop}_{sm}(X \xrightarrow f Y) \to \Cal  Z^*(X,Y)$$
 defined by
$\frak f ([V \xrightarrow p Y; L_1, \cdots, L_r]):= [X \xleftarrow p V \xrightarrow {f \circ p} Y; L_1, \cdots, L_r].$
This ``forget" map is compatible with the bivariant product $\bullet$, the bivariant pushforward and the Chern class operator, but not necessarily with the pullback. 
As to correspondences, D. Gaitsgory and N. Rozenblyum study \emph{correspondences in derived algebraic geometry} intensively in their recent book \cite{GR1} (cf. \cite{GR2}) . For example they consider the category $Corr ({\bf C})_{vert, horiz}$ for a category $\bf C$ equipped with two classes of morphisms $vert$ and $horiz$ (both closed under composition) such that 
\begin{enumerate}
\item the objets of $Corr ({\bf C})_{vert, horiz}$ are the same as those of $\bf C$ and 
\item the morphisms are correspondences, i.e., a morphism from $c_0$ to $c_1$ is a correspondence (drawn as follows in \cite{GR1}):
$$\xymatrix{
c_{0,1}\ar[d]_f \ar[r]^g & c_0\\
c_1
}
$$
\end{enumerate}
where $f$ is $vert$ and $g$ is $horiz$. If we use our notation, $Corr ({\bf C})_{vert, horiz}$ can be denoted by
$Corr ({\bf C})^{horiz}_{vert}$ and the above diagram is $c_0 \xleftarrow g c_{0,1} \xrightarrow f c_1$.
So, in this way of thinking of two kinds of bivariant theories, it remains to see whether one could get a ``correspondence" version of Annala's bivariant derived algebraic cobordism, i.e., whether one could consider some reasonable relations $\Cal R(X, Y)$ on $\Cal  Z^{prop}_{qusm}(X,Y)$ which is a derived algebraic geometric version of $\Cal  Z^*(X,Y)$ with \emph{smooth morphism} being replaced by \emph{quasi-smooth morphism}, 
such that the following diagram commutes:
$$\xymatrix{
\mathbb {OM}_{qusm}^{prop}(X \xrightarrow f Y) \ar[d]_{\pi} \ar[r]^{\frak f}  & \Cal  Z^{prop}_{qusm}(X,Y) \ar[d]_{\pi}\\
\frac{\mathbb {OM}_{qusm}^{prop}(X \xrightarrow f Y)}{\Cal R(X \xrightarrow f Y)}=: \Omega^*(X \xrightarrow f Y)  \ar[r]_{\frak f} &  \Cal Z^{\bigstar}(X,Y) := \frac{\Cal  Z^{prop}_{qusm}(X,Y)}{\Cal R(X, Y)}.
}
$$
We hope to be able to treat this problem in a different paper. It would be nice that in derived algebraic geometry there is some relation between Annala's bivariant derived algebraic cobordism $\Omega^*(X \xrightarrow f Y)$ and Kasparov's bivariant $KK$-theory $KK(X,Y)$ via the ``forget"  map $\frak f: \Omega^*(X \xrightarrow f Y) \to KK(X,Y)$
for certain reasonable maps $f:X \to Y$.
\end{rem}

\section {Fulton--MacPherson's bivariant theory}\label{FM-BT}

We make a quick review of Fulton--MacPherson's bivariant theory \cite {FM} (also see \cite{Fulton-book}) (cf. a universal bivariant theory \cite{Yokura-obt, Yokura-obt2}). 

Let $\Cal V$ be a category which has a final object $pt$ and on which the fiber product or fiber square is well-defined. Also we consider a class of maps, called ``confined maps" (e.g., proper maps, projective maps, in algebraic geometry), which are \emph{closed under composition and base change and contain all the identity maps}, and a class of fiber squares, called ``independent squares" (or ``confined squares", e.g., ``Tor-independent" in algebraic geometry, a fiber square with some extra conditions required on morphisms of the square), which satisfy the following:

(i) if the two inside squares in  
$$\CD
X''@> {h'} >> X' @> {g'} >> X \\
@VV {f''}V @VV {f'}V @VV {f}V\\
Y''@>> {h} > Y' @>> {g} > Y \endCD
\quad \quad \qquad \text{or} \qquad \quad \quad 
\CD
X' @>> {h''} > X \\
@V {f'}VV @VV {f}V\\
Y' @>> {h'} > Y \\
@V {g'}VV @VV {g}V \\
Z'  @>> {h} > Z \endCD
$$
are independent, then the outside square is also independent,

(ii) any square of the following forms are independent:
$$
\xymatrix{X \ar[d]_{f} \ar[r]^{\op {id}_X}&  X \ar[d]^f & & X \ar[d]_{\op {id}_X} \ar[r]^f & Y \ar[d]^{\op {id}_Y} \\
Y \ar[r]_{\op {id}_X}  & Y && X \ar[r]_f & Y}
$$
where $f:X \to Y$ is \emph{any} morphism. 

A bivariant theory $\bB$ on a category $\Cal V$ with values in the category of graded abelian groups is an assignment to each morphism
$ X  \xrightarrow{f} Y$
in the category $\Cal V$ a graded abelian group (in most cases we ignore the grading )
$\bB(X  \xrightarrow{f} Y)$
which is equipped with the following three basic operations. The $i$-th component of $\bB(X  \xrightarrow{f} Y)$, $i \in \bZ$, is denoted by $\bB^i(X  \xrightarrow{f} Y)$.
\begin{enumerate}
\item {\bf Product}: For morphisms $f: X \to Y$ and $g: Y
\to Z$, the product operation
$$\bullet: \bB^i( X  \xrightarrow{f}  Y) \otimes \bB^j( Y  \xrightarrow{g}  Z) \to
\bB^{i+j}( X  \xrightarrow{gf}  Z)$$
is  defined.

\item {\bf Pushforward}: For morphisms $f: X \to Y$
and $g: Y \to Z$ with $f$ \emph {confined}, the pushforward operation
$$f_*: \bB^i( X  \xrightarrow{gf} Z) \to \bB^i( Y  \xrightarrow{g}  Z) $$
is  defined.

\item {\bf Pullback} : For an \emph{independent} square \qquad $\CD
X' @> g' >> X \\
@V f' VV @VV f V\\
Y' @>> g > Y, \endCD
$

the pullback operation
$$g^* : \bB^i( X  \xrightarrow{f} Y) \to \bB^i( X'  \xrightarrow{f'} Y') $$
is  defined.
\end{enumerate}

These three operations are required to satisfy the following seven compatibility 
axioms (\cite [Part I, \S 2.2]{FM}):

\begin{enumerate}
\item[($A_1$)] {\bf Product is associative}: for $X \xrightarrow f Y  \xrightarrow g Z \xrightarrow h  W$ with $\alp \in \bB(X \xrightarrow f Y),  \be \in \bB(Y \xrightarrow g Z), \ga \in \bB(Z \xrightarrow h W)$,
$$(\alp \bullet\be) \bullet \ga = \alp \bullet (\be \bullet \ga).$$
\item[($A_2$)] {\bf Pushforward is functorial} : for $X \xrightarrow f Y  \xrightarrow g Z \xrightarrow h  W$ with $f$ and $g$ confined and $\alp \in \bB(X \xrightarrow {h\circ g\circ f} W)$
$$(g\circ f)_* (\alp) = g_*(f_*(\alp)).$$
\item[($A_3$)] {\bf Pullback is functorial}: given independent squares
$$\CD
X''@> {h'} >> X' @> {g'} >> X \\
@VV {f''}V @VV {f'}V @VV {f}V\\
Y''@>> {h} > Y' @>> {g} > Y \endCD
$$
$$(g \circ h)^* = h^* \circ g^*.$$
\item[($A_{12}$)] {\bf Product and pushforward commute}: for $X \xrightarrow f Y  \xrightarrow g Z \xrightarrow h  W$ with $f$ confined and $\alp \in \bB(X \xrightarrow {g \circ f} Z),  \be \in \bB(Z \xrightarrow h W)$,
$$f_*(\alp \bullet\be)  = f_*(\alp) \bullet \be.$$
\item[($A_{13}$)] {\bf Product and pullback commute}: given independent squares
$$\CD
X' @> {h''} >> X \\
@V {f'}VV @VV {f}V\\
Y' @> {h'} >> Y \\
@V {g'}VV @VV {g}V \\
Z'  @>> {h} > Z \endCD
$$
with $\alp \in \bB(X \xrightarrow {f} Y),  \be \in \bB(Y \xrightarrow g Z)$,
$$h^*(\alp \bullet\be)  = {h'}^*(\alp) \bullet h^*(\be).$$
\item[($A_{23}$)] \label{push-pull}{\bf Pushforward and pullback commute}: given independent squares
$$\CD
X' @> {h''} >> X \\
@V {f'}VV @VV {f}V\\
Y' @> {h'} >> Y \\
@V {g'}VV @VV {g}V \\
Z'  @>> {h} > Z \endCD
$$
with $f$ confined and $\alp \in \bB(X \xrightarrow {g\circ f} Z)$,
$$f'_*(h^*(\alp))  = h^*(f_*(\alp)).$$
\item[($A_{123}$)] {\bf Projection formula}: given an independent square with $g$ confined and $\alp \in \bB(X \xrightarrow {f} Y),  \be \in \bB(Y' \xrightarrow {h \circ g} Z)$
$$\CD
X' @> {g'} >> X \\
@V {f'}VV @VV {f}V\\
Y' @>> {g} > Y @>> h >Z \\
\endCD
$$
and $\alp \in \bB(X \xrightarrow {f} Y),  \be \in \bB(Y' \xrightarrow {h \circ g} Z)$,
$$g'_*(g^*(\alp) \bullet \be)  = \alp \bullet g_*(\be).$$
\end{enumerate}

We also assume that $\bB$ has units:

\underline {Units}: $\bB$ has units, i.e., there is an element $1_X \in \bB^0( X  \xrightarrow{\op {id}_X} X)$ such that $\alp \bullet 1_X = \alp$ for all morphisms $W \to X$ and all $\alp \in \bB(W \to X)$, such that $1_X \bullet \beta = \beta $ for all morphisms $X \to Y$ and all $\beta \in \bB(X \to Y)$, and such that $g^*1_X = 1_{X'}$ for all $g: X' \to X$. 

\underline {Commutativity}\label{commutativity}: $\bB$ is called \emph{commutative} if whenever both
$$\CD
W @> {g'} >> X \\
@V {f'}VV @VV {f}V\\
Y @>> {g} > Z  \\
\endCD  
\quad \quad \text{and} \quad \quad 
\CD
W @> {f'} >> Y \\
@V {g'}VV @VV {g}V\\
X @>> {g} > Z \\
\endCD  
$$
are independent squares with $\alp \in \bB(X \xrightarrow f Z)$ and $\be \in \bB(Y \xrightarrow g Z)$,
$$g^*(\alp) \bullet \be = f^*(\be) \bullet \alp .$$
Let $\bB, \bB'$ be two bivariant theories on a category $\Cal V$. A {\it Grothendieck transformation} from $\bB$ to $\bB'$, $\ga : \bB \to \bB'$
is a collection of homomorphisms
$\bB(X \to Y) \to \bB'(X \to Y)$
for a morphism $X \to Y$ in the category $\Cal V$, which preserves the above three basic operations: 
\begin{enumerate}
\item $\ga (\alp \bullet_{\bB} \be) = \ga (\alp) \bullet _{\bB'} \ga (\be)$, 
\item $\ga(f_{*}\alp) = f_*\ga (\alp)$, and 
\item $\ga (g^* \alp) = g^* \ga (\alp)$. 
\end{enumerate}
A bivariant theory unifies both a covariant theory and a contravariant theory in the following sense:

$\bB_*(X):= \bB(X \to pt)$ becomes a covariant functor for {\it confined}  morphisms and 

$\bB^*(X) := \bB(X  \xrightarrow{id}  X)$ becomes a contravariant functor for {\it any} morphisms. 
\noindent
A Grothendieck transformation $\ga: \bB \to \bB'$ induces natural transformations $\ga_*: \bB_* \to \bB_*'$ and $\ga^*: \bB^* \to {\bB'}^*$.

\begin{defn}\label{grading}
As to the grading, $\bB_i(X):= \bB^{-i}(X  \to pt)$ and
$\bB^j(X):= \bB^j(X  \xrightarrow{id}  X)$.
\end{defn}
 
\begin{defn}\label{canonical}(\cite[Part I, \S 2.6.2 Definition]{FM}) Let $\Cal S$ be a class of maps in $\Cal V$, which is closed under compositions and containing all identity maps. Suppose that to each $f: X \to Y$ in $\Cal S$ there is assigned an element
$\theta(f) \in \bB(X  \xrightarrow {f} Y)$ satisfying that
\begin{enumerate}
\item [(i)] $\theta (g \circ f) = \theta(f) \bullet \theta(g)$ for all $f:X \to Y$, $g: Y \to Z \in \Cal S$ and
\item [(ii)] $\theta(\op {id}_X) = 1_X $ for all $X$ with $1_X \in \bB^*(X):= \bB^*(X  \xrightarrow{\op {id}_X} X)$ the unit.
\end{enumerate}
Then $\theta(f)$ is called a {\it orientation} of $f$. (In \cite[Part I, \S 2.6.2 Definition]{FM} it is called a {\it canonical orientation} of $f$, but in this paper it shall be simply called an orientation.)
\end{defn} 

\begin{defn} Let $\Cal S$ be another class of maps called ``specialized maps" (e.g., smooth maps in algebraic geometry) in $\Cal V$ , which is closed under composition, closed under base change and containing all identity maps. Let $\bB$ be a  
bivariant theory. If $\Cal S$ has  an orientation $\theta$ for $\bB$ and it satisfies 
that for an independent square with $f \in \Cal S$
$$
\CD
X' @> g' >> X\\
@Vf'VV   @VV f V \\
Y' @>> g > Y
\endCD
$$
the following condition holds: 
$\theta (f') = g^* \theta (f)$, 
(which means that the orientation $\theta$ is preserved by  the pullback operation), then we call $\theta$ a {\it stable orientation} and say that $\Cal S$ is {\it stably $\bB$-oriented}.
\end{defn}

\section{Oriented bivariant theory and a universal oriented bivariant theory }

Levine--Morel's algebraic cobordism is the universal one among the so-called \emph {oriented} Borel--Moore functors with products for algebraic schemes. Here ``oriented" means that the given Borel--Moore functor $H_*$ is equipped with the endomorphism $\tilde c_1(L): H_*(X) \to H_*(X)$ for a line bundle $L$ over the scheme $X$. Motivated by this ``orientation" (which is different from the one given in Definition \ref{canonical}, but we still call this ``orientation" using a different symbol so that the reader will not be confused with terminologies), in \cite[\S4]{Yokura-obt} we introduce an orientation to bivariant theories for any category, using the notion of \emph {fibered categories} in abstract category theory (e.g, see \cite{Vistoli}) and such a bivariant theory equipped with such an orientation (Chern class operator) is called \emph{an oriented bivariant theory}. 

\begin{defn}(\cite[Definition 4.2]{Yokura-obt}) (an oriented bivariant theory) \label{orientation} Let $\bB$ be a bivariant theory on a category $\Cal V$.

\begin{enumerate}
\item For a fiber-object $L$ over $X$, the \emph{``operator" on $\bB$ associated to $L$}, denoted by $\phi(L)$, is defined to be an endomorphism
$$\phi(L): \bB(X  \xrightarrow{f}  Y) \to \bB(X  \xrightarrow{f}  Y) $$
which satisfies the following properties:

(O-1) {\bf identity}: If $L$ and $L'$ are two fiber-objects over $X$ and isomorphic (i.e., if $f:L\to X$ and $f':L' \to X$, then there exists an isomorphism $i: L\to L'$ such that $f = f' \circ i$) , then we have
$$\phi(L) = \phi(L'): \bB(X  \xrightarrow{f}  Y) \to \bB(X  \xrightarrow{f}  Y).$$

(O-2)  {\bf commutativity}: Let $L$ and $L'$ be two fiber-objects  over $X$, then we have
$$\phi(L) \circ \phi(L') = \phi(L') \circ \phi(L) :\bB(X  \xrightarrow{f}  Y) \to \bB(X  \xrightarrow{f}  Y). $$

(O-3)   {\bf compatibility with product}: For morphisms $f:X \to Y$ and $g:Y \to Z$,  $\alp \in \bB(X  \xrightarrow{f} Y)$ and $ \be \in \bB(Y  \xrightarrow{g} Z)$,  a fiber-object  $L$ over $X$ and a fiber-object  $M$ over $Y$, we have
 $$ \phi(L) (\alp \bullet \be) = \phi(L)(\alp) \bullet \be, \quad  \phi(f^*M) (\alp \bullet \be) = \alp \bullet \phi(M)(\be).$$

(O-4)  {\bf compatibility with pushforward}: For a confined morphism $f:X \to Y$ and a fiber-object $M$ over $Y$ we have 
$$ f_*\left (\phi(f^*M)(\alp) \right ) = \phi(M)(f_*\alp).$$

(O-5)   {\bf compatibility with pullback}: For an independent square and a fiber-object  $L$ over $X$
$$
\CD
X' @> g' >> X \\
@V f' VV @VV f V\\
Y' @>> g > Y 
 \endCD
$$
we have 
$$g^*\left (\phi(L)(\alp) \right ) = \phi({g'}^*L)(g^*\alp).$$

The above operator is called an ``{\it orientation}" and a bivariant theory equipped with such an orientation is called an {\it oriented bivariant theory}, denoted by $\bOB$. 
\item An {\it oriented Grothendieck transformation} between two oriented bivariant theories is a Grothendieck transformation which preserves or is compatible with the operator, i.e., for two oriented bivariant theories $\bOB$ with an orientation $\phi$ and $\bOB'$ with an orientation $\phi'$ the following diagram commutes
$$
\CD
\bOB (X  \xrightarrow{f}  Y)  @> {\phi(L)}>> \bOB (X  \xrightarrow{f}  Y) \\
@V \ga VV @VV \ga V\\
\bOB' (X  \xrightarrow{f}  Y) @>>{\phi'(L)} > \bOB' (X  \xrightarrow{f}  Y).
 \endCD
$$
\end{enumerate}
\end{defn} 

\begin{thm} (\cite[Theorem 4.6]{Yokura-obt}) \label{obt}(A universal oriented bivariant theory)
 Let  $\Cal V$ be a category with a class $\Cal C$ of confined morphisms, a class of independent squares, a class  $\Cal S$ of specialized morphisms and $\Cal L$ a fibered category over $\Cal V$.  We define 
$$\bOM^{\Cal C} _{\Cal S}(X  \xrightarrow{f}  Y)$$
to be the free abelian group generated by the set of isomorphism classes of cobordism cycles over $X$
$$[V  \xrightarrow{h} X; L_1, L_2, \cdots, L_r]$$
such that $h \in \Cal  C$, $f \circ h: W \to Y \in \Cal S$ and $L_i$ a fiber-object over $V$.
\begin{enumerate}
\item The association $\bOM^{\Cal C} _{\Cal S}$ becomes an oriented bivariant theory if the four operations are defined as follows:
\begin{enumerate}

\item  {\bf Orientation $\Phi$}: For a morphism $f:X \to Y$ and a fiber-object $L$ over $X$, the operator
$$\Phi (L):\bOM^{\Cal C} _{\Cal S} ( X  \xrightarrow{f}  Y) \to \bOM^{\Cal C} _{\Cal S} ( X  \xrightarrow{f}  Y) $$
is defined by
$$\Phi(L)([V  \xrightarrow{h} X; L_1, L_2, \cdots, L_r]):=[V  \xrightarrow{h} X; L_1, L_2, \cdots, L_r, h^*L].$$
and extended linearly.

\item {\bf Product}: For morphisms $f: X \to Y$ and $g: Y
\to Z$, the product operation
$$\bullet: \bOM^{\Cal C} _{\Cal S} ( X  \xrightarrow{f}  Y) \otimes \bOM^{\Cal C} _{\Cal S} ( Y  \xrightarrow{g}  Z) \to
\bOM^{\Cal C} _{\Cal S} ( X  \xrightarrow{gf}  Z)$$
is  defined as follows: The product is defined by
\begin{align*}
& [V  \xrightarrow{h} X; L_1, \cdots, L_r]  \bullet [W  \xrightarrow{k} Y; M_1, \cdots, M_s] \\
& :=  [V'  \xrightarrow{h \circ {k}''}  X; {{k}''}^*L_1, \cdots,{{k}''}^*L_r, (f' \circ {h}')^*M_1, \cdots, (f' \circ {h}')^*M_s ]\
\end{align*}
and extended bilinearly. Here we consider the following fiber squares
$$\CD
V' @> {h'} >> X' @> {f'} >> W \\
@V {{k}''}VV @V {{k}'}VV @V {k}VV\\
V@>> {h} > X @>> {f} > Y @>> {g} > Z .\endCD
$$
\item {\bf Pushforward}: For morphisms $f: X \to Y$
and $g: Y \to Z$ with $f$ confined, the pushforward operation
$$f_*: \bOM^{\Cal C} _{\Cal S} ( X  \xrightarrow{gf} Z) \to \bOM^{\Cal C} _{\Cal S} ( Y  \xrightarrow{g}  Z) $$
is  defined by
$$f_*\left ([V  \xrightarrow{h} X; L_1, \cdots, L_r]  \right) := [V  \xrightarrow{f \circ h} Y; L_1, \cdots, L_r]$$
and extended linearly.

\item {\bf Pullback}: For an independent square
$$\CD
X' @> g' >> X \\
@V f' VV @VV f V\\
Y' @>> g > Y, \endCD
$$
the pullback operation
$$g^*: \bOM^{\Cal C} _{\Cal S} ( X  \xrightarrow{f} Y) \to \bOM^{\Cal C} _{\Cal S}( X'  \xrightarrow{f'} Y') $$
is  defined by
$$g^*\left ([V  \xrightarrow{h} X; L_1, \cdots, L_r] \right):=  [V'  \xrightarrow{{h}'}  X'; {g''}^*L_1, \cdots, {g''}^*L_r]$$
and extended linearly, where we consider the following fiber squares:
$$\CD
V' @> g'' >> V \\
@V {h'} VV @VV {h}V\\
X' @> g' >> X \\
@V f' VV @VV f V\\
Y' @>> g > Y. \endCD
$$
\end{enumerate}

\item Let $\Cal {OBT}$ be a class of oriented bivariant theories $\bOB$ on the same category $\Cal V$ with a class $\Cal C$ of confined morphisms, a class of independent squares, a class $\Cal S$ of specialized morphisms and a fibered category $\Cal L$ over $\Cal V$. Let $\Cal S$ be stably $\bOB$-oriented for any oriented bivariant theory $\bOB \in \Cal {OBT}$. Then, for each oriented bivariant theory $\bOB \in \Cal {OBT}$ with an orientation $\phi$  there exists a unique oriented Grothendieck transformation
$$\ga_{\bOB} : \bOM^{\Cal C} _{\Cal S} \to \bOB$$
such that for any $f: X \to Y \in \Cal S$ the homomorphism
$\ga_{\bOB} : \bOM^{\Cal C} _{\Cal S}(X  \xrightarrow{f}  Y) \to \bOB(X  \xrightarrow{f}  Y)$
satisfies the normalization condition that $$\ga_{\bOB}([X  \xrightarrow{\op {id}_X}  X; L_1, \cdots, L_r]) = \phi(L_1) \circ \cdots \circ \phi(L_r) (\theta_{\bOB}(f)).$$
\end{enumerate}
\end{thm}
In this paper we consider the category $\Cal V$ of complex algebraic varieties or schemes and we consider proper morphisms for the class $\Cal C$ of confined maps, smooth morphisms for the class $\Cal S$ of specialized morphisms, fiber squares for independent squares, and line bundles for a fibered category $\Cal L$ over $\Cal V$. So,  $\bOM^{\Cal C} _{\Cal S}(X  \xrightarrow{f}  Y)$ shall be denote by $\bOM^{prop} _{sm}(X  \xrightarrow{f}  Y).$ For a smooth morphism $f:X \to Y$
$$\theta (f):= [X \xrightarrow {\op{id}_X} X] \in \bOM^{prop} _{sm}(X  \xrightarrow{f}  Y) $$
is clearly a stable orientation. As mentioned in the introduction, we can consider the above free abelian group $\bOM^{prop} _{sm}(X  \xrightarrow{f}  Y)$ modulo the additive relation (\ref{add}), i.e., the Grothendieck group of the monoid of the isomorphism classes of cobordism bicycles of finite sets of line bundles, which is denoted by the same notation $\bOM^{prop} _{sm}(X  \xrightarrow{f}  Y)$.
\section{Cobordism bicycles of vector bundles}
In this section we consider extending the notion of algebraic cobordism of vector bundles due to Y.-P. Lee and R. Pandharipande \cite{LeeP} (cf. \cite{LTY}) to correspondences.

\begin{defn} Let $X \xleftarrow p V \xrightarrow s Y$ be a correspondence (sometimes called a \emph{span} or a \emph{roof}) such that $p:V \to X$ is a \emph{proper} map and $s: V \to Y$ be a \emph{smooth} map, and let $E$ be a complex vector bundle. Then
$$(X \xleftarrow p V \xrightarrow s Y; E)$$
is called a \emph{cobordism bicycle of a vector bundle}.
\end{defn}
\begin{rem} 
The above correspondence $X \xleftarrow p V \xrightarrow s Y$ shall be called a proper-smooth correspondence, abusing words.
Mimicking naming used in \cite{BB}, \cite{LM} and \cite{LeeP}, the above proper-smooth correspondence equipped with a vector bundle is simply also named ``cobordism bicycle of a vector bundle.
\end{rem}
\begin{defn}\label{bicycle} Let $(X \xleftarrow p V \xrightarrow s Y; E)$ and $(X \xleftarrow {p'} V' \xrightarrow {s'} Y; E')$ be two cobordism bicycles of vector bundles of the same rank. If there exists an isomorphism $h: V \cong V'$ such that 
\begin{enumerate}
\item $(X \xleftarrow p V \xrightarrow s Y) \cong (X \xleftarrow {p'} V' \xrightarrow {s'} Y)$ as correspondences, i.e., the following diagrams commute:
$$\xymatrix{
& V\ar [dl]_{p} \ar[dr]^{s} \ar[dd]_h &\\
X   && Y \\
 & V'\ar[ul]^{p'} \ar[ur]_{s'}&}
$$

\item $E \cong h^*E'$,
\end{enumerate}
then they are called isomorphic and denoted by
$$(X \xleftarrow p V \xrightarrow s Y; E) \cong (X \xleftarrow {p'} V' \xrightarrow {s'} Y; E').$$
\end{defn}
The isomorphism class of a cobordism bicycle of a vector bundle $(X \xleftarrow p V \xrightarrow s Y; E)$ is denoted by $[X \xleftarrow p V \xrightarrow s Y; E]$, which is still called a cobordism bicycle of a vector bundle.  For a fixed rank $r$ for vector bundles, the set of isomorphism classes of cobordism bicycles of vector bundles for a pair $(X,Y)$  becomes a commutative monoid by the disjoint sum:
\begin{align*}
[X \xleftarrow {p_1} V_1 \xrightarrow {s_1} Y; E_1] + [X \xleftarrow {p_2} & V_2 \xrightarrow {s_2} Y; E_2]\\
& := [X \xleftarrow {p_1+p_2} V_1 \sqcup V_2 \xrightarrow {s_1+s_2} Y; E_1 + E_2],
\end{align*}
where $E_1 + E_2$ is a vector bundle such that $(E_1 + E_2)|_{V_1} =E_1$ and $(E_1 + E_2)|_{V_2} =E_2$.
This monoid is denoted by $\Cal M_r(X,Y)$ and another grading of $[X \xleftarrow p V \xrightarrow s Y; E]$ is defined by the relative dimension of the smooth map $s$, denoted by $\op{dim} s$, thus by double grading, $[X \xleftarrow p V \xrightarrow s Y; E] \in \Cal   M_{n,r}(X,Y)$ means that $n = \op{dim} s$ and $r = \op{rank} E$.
The group completion of this monoid, i.e., the Grothendieck group, is denoted by $\Cal   M_{n,r}(X,Y)^+$. We use this notation, mimicking \cite{LM}.
\begin{rem} For a fixed rank $r$, $\Cal   M_{*, r}(X,Y)^+ = \bigoplus \Cal   M_{n,r}(X,Y)^+$ is  a graded abelian group.
\end{rem}
\begin{rem} When $Y=pt$ a point, $\Cal   M_{n,r}(X,pt)^+$ is nothing but $\Cal   M_{n,r}(X)^+$ considered in Lee--Pandharipande \cite{LeeP}. In this sense, when $X = pt$ a point, $\Cal   M_{n,r}(pt,Y)^+$ is a new object to be investigated.
\end{rem}

\begin{defn}[product of cobordism bicycles]\label{prod} For three varieties $X,Y,Z$, we define the following two kinds of product $\bullet_{\oplus}$ and $\bullet_{\otimes}$:
\begin{enumerate}
\item (by the Whitney sum $\oplus$)
$$\bullet_{\oplus} : \Cal   M_{m,r}(X,Y)^+ \otimes \Cal   M_{n,k}(Y,Z)^+ \to \Cal   M_{m+n,r+k}(X,Z)^+; $$
\begin{align*}
[X \xleftarrow p  V \xrightarrow s Y; E] \bullet_{\oplus} [Y \xleftarrow q & W \xrightarrow t Z; F]\\
& : = [(X \xleftarrow p V \xrightarrow s Y) \circ (Y \xleftarrow q W \xrightarrow s Z);\widetilde{q}^*E \oplus \widetilde{s}^*F ],
\end{align*}
\item (by the tensor product $\otimes$)
$$\bullet_{\otimes} : \Cal   M_{m,r}(X,Y)^+ \otimes \Cal   M_{n,k}(Y,Z)^+ \to \Cal   M_{m+n,rk}(X,Z)^+;$$
\begin{align*}
[X \xleftarrow p  V \xrightarrow s Y; E] \bullet_{\otimes} [Y \xleftarrow q  & W \xrightarrow t Z; F]\\
& : = [(X \xleftarrow p V \xrightarrow s Y) \circ (Y \xleftarrow q W \xrightarrow s Z);\widetilde{q}^*E \otimes \widetilde{s}^*F ].
\end{align*}
\end{enumerate}
Here we consider the following commutative diagram
$$\xymatrix{
& & \widetilde{q}^*E \oplus \widetilde{s}^*F \quad \text{or} \quad \widetilde{q}^*E \otimes  \widetilde{s}^*F \ar[d] \\
& E \ar[d]& V\times_Y W\ar [dl]_{\widetilde{q}} \ar[dr]^{\widetilde{s}} & F \ar[d]&\\
& V \ar [dl]_{p} \ar [dr]^{s} && W \ar [dl]_{q} \ar[dr]^{t}\\
X & &  Y && Z }
$$
\end{defn}

\begin{lem}
The products $\bullet_{\oplus}$ and $\bullet_{\otimes}$ are both bilinear.
\end{lem}
\begin{rem} $\Cal   M_{*,*}(X,X)^+$ is a double graded commutative ring with respect to both products $\bullet_{\oplus}$ and $\bullet_{\otimes}$.
\end{rem}
\begin{rem} We consider the above product $\bullet_{\oplus}$ for $Y = Z=pt$ a point. Since $\Cal   M_{n,r}(X,pt)^+ = \Cal   M_{n,r}(X)^+$ and $\Cal   M_{n,r}(pt,pt)^+ = \Cal   M_{n,r}(pt)^+$, we have
$$\bullet_{\oplus}: \Cal   M_{m,r}(X)^+ \otimes \Cal   M_{n,k}(pt)^+ \to \Cal   M_{m+n,r+k}(X)^+.$$
\begin{align*}[X \xleftarrow p  V \xrightarrow s pt; E] & \bullet_{\oplus} [pt \xleftarrow q W \xrightarrow t pt; F]\\
&= [(X \xleftarrow p V \xrightarrow s pt) \circ (pt \xleftarrow q W \xrightarrow s pt);(pr_1)^*E \oplus (pr_2)^*F ],
\end{align*}
which is rewritten as follows, by using the notations used in \cite{LeeP}:
$$[V \xrightarrow p X, E] \bullet_{\oplus} [W; F]= [V \times W \xrightarrow {p \circ pr_1} X; (pr_1)^*E \oplus (pr_2)^*F].$$
$$\xymatrix{
& & (pr_1)^*E \oplus (pr_2)^*F \ar[d] \\
& E \ar[d]& V\times W\ar [dl]_{pr_1} \ar[dr]^{pr_2} & F \ar[d]&\\
& V \ar [dl]_{p} \ar [dr]^{s} && W \ar [dl]_{q} \ar[dr]^{t}\\
X & &  pt && pt }
$$

In \cite[\S 0.8]{LeeP} they define the following map (which turns out to be isomorphic \cite[Theorem 3]{LeeP})
$$\gamma_X: \omega_*(X) \otimes_{\omega_*(pt)} \omega_{*,r}(pt) \to \omega_{*,r}(X)$$
of $\omega_*(pt)$-modules by
\begin{align*}
\gamma_X \Bigl ( [Y \xrightarrow f X] \otimes [\mathbb P^{\lambda}, \Cal  O^{r- \ell(\lambda)} \oplus & \bigoplus_{m \in \lambda} L_m] \Bigr ) \\
&:= [Y \times \mathbb P^{\lambda} \xrightarrow {f \circ p_Y} X, \Cal  O^{r- \ell(\lambda)} \oplus \bigoplus_{m \in \lambda} p^*_{\mathbb P^{\lambda}}L_m].
\end{align*}
In fact this map is nothing but our product $\bullet_{\oplus}$ at least at the level of $\Cal   M_{*,*}$: 
$$\bullet_{\oplus}: \Cal   M_{*,0}(X,pt) \otimes \Cal   M_{*,r}(pt,pt) \to  \Cal   M_{*,r}(X,pt).$$ 
Note that $\Cal   M_{*,0}(X,pt) =\Cal   M_{*,0}(X)$,  $\Cal   M_{*,r}(pt,pt)=\Cal   M_{*,r}(pt)$ and $\Cal   M_{*,r}(X,pt) = \Cal   M_{*,r}(X).$
Using our notation, we have
\begin{itemize}
\item $[Y \xrightarrow f X]=[X \xleftarrow {f} Y \xrightarrow {s} pt]$,
\item $[\mathbb P^{\lambda}, \Cal  O^{r- \ell(\lambda)} \oplus \bigoplus_{m \in \lambda} L_m] = [pt \xleftarrow {} \mathbb P^{\lambda} \xrightarrow {} pt; \Cal  O^{r- \ell(\lambda)} \oplus \bigoplus_{m \in \lambda} L_m]$,
\item $[Y \times \mathbb P^{\lambda} \xrightarrow {f \circ p_Y} X, \Cal  O^{r- \ell(\lambda)} \oplus \bigoplus_{m \in \lambda} p^*_{\mathbb P^{\lambda}}L_m] = [X \xleftarrow {f \circ p_Y} Y \times \mathbb P^{\lambda} \xrightarrow {} pt; \Cal  O^{r- \ell(\lambda)} \oplus \bigoplus_{m \in \lambda} p^*_{\mathbb P^{\lambda}}L_m]$.
\end{itemize}
Our product $\bullet_{\oplus}$ gives us
\begin{align*}
[X \xleftarrow {f} Y \xrightarrow {s} pt] \bullet_{\oplus} & [pt \xleftarrow {} \mathbb P^{\lambda} \xrightarrow {} pt;  \Cal  O^{r- \ell(\lambda)} \oplus \bigoplus_{m \in \lambda} L_m] \\
& = [X \xleftarrow {f \circ p_Y} Y \times \mathbb P^{\lambda} \xrightarrow {} pt; \Cal  O^{r- \ell(\lambda)} \oplus \bigoplus_{m \in \lambda} p^*_{\mathbb P^{\lambda}}L_m].
\end{align*}
Indeed for this product we consider the following diagram:
$$
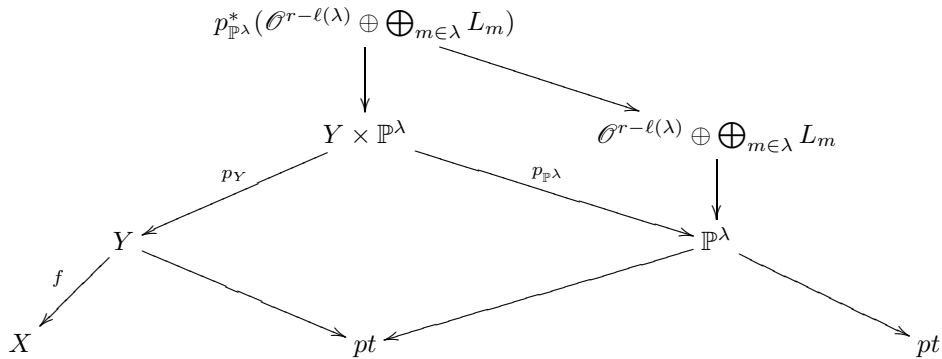
\xymatrix{
& & p_{\mathbb P^{\lambda}}^*(\Cal  O^{r- \ell(\lambda)} \oplus \bigoplus_{m \in \lambda} L_m) \ar[d] \ar[dr] \\
& & Y \times \mathbb P^{\lambda} \ar [dl]_{p_Y} \ar[dr]^{p_{\mathbb P^{\lambda}}} & \Cal  O^{r- \ell(\lambda)} \oplus \bigoplus_{m \in \lambda} L_m \ar[d]&\\
& Y \ar [dl]_{f} \ar [dr]^{} && \mathbb P^{\lambda} \ar [dl]_{} \ar[dr]^{}\\
X & &  pt && pt }
$$
Note that $p_{\mathbb P^{\lambda}}^*(\Cal  O^{r- \ell(\lambda)} \oplus \bigoplus_{m \in \lambda} L_m) = \Cal  O^{r- \ell(\lambda)} \oplus \bigoplus_{m \in \lambda} p_{\mathbb P^{\lambda}}^*L_m$ since $\Cal  O^{r- \ell(\lambda)}$ is a trivial bundle. Therefore we can see that
\begin{align*}
\gamma_X\Bigl( [Y \xrightarrow f X] \otimes [\mathbb P^{\lambda}, & \Cal  O^{r- \ell(\lambda)} \oplus \bigoplus_{m \in \lambda} L_m] \Bigr) \\
& = [X \xleftarrow {f} Y \xrightarrow {s} pt] \bullet_{\oplus}  [pt \xleftarrow {} \mathbb P^{\lambda} \xrightarrow {} pt;  \Cal  O^{r- \ell(\lambda)} \oplus \bigoplus_{m \in \lambda} L_m].
\end{align*}
\end{rem}
\begin{rem} Let $[X \xleftarrow p  V \xrightarrow s Y; E] \in \Cal   M_{m,r}(X,Y)^+_{\oplus}$, and let
\begin{enumerate}
\item $[X \xleftarrow p V] := [X \xleftarrow p V \xrightarrow {\op{id}_V} V] \in \Cal   M_{0,0}(X,V)^+_{\oplus}$,
\item $[V; E] := [V \xleftarrow {\op{id}_V} V \xrightarrow {\op{id}_V} V; E] \in \Cal   M_{0,r}(V,V)^+_{\oplus}$,
\item $[V \xrightarrow s Y]:= [V \xleftarrow {\op{id}_V} V \xrightarrow s Y] \in \Cal   M_{m,0}(V,Y)^+_{\oplus}$.
\end{enumerate}
Then we have $[X \xleftarrow p  V \xrightarrow s Y; E] = [X \xleftarrow p V]\bullet_{\oplus} [V; E] \bullet_{\oplus}[V \xrightarrow s Y].$
\end{rem}
Now we define pushforward and pullback of cobordism bicycles of vector bundles:
\begin{defn}\label{push-pull}
\begin{enumerate} 
\item (Pushforward) 
\begin{enumerate}
\item For a \emph{proper} map $f:X \to X'$, the (proper) pushforward \emph{acting on the first factor $X \xleftarrow{p} V$}, in a usual way denoted by $f_*:\Cal   M_{m,r}(X,Y)^+ \to \Cal M_{m,r}(X',Y)^+$,  is defined by
$$f_*([X \xleftarrow{p} V  \xrightarrow{s} Y; E]):= [X' \xleftarrow{f \circ p} V  \xrightarrow{s} Y; E].$$
\item For a \emph{smooth} map $g:Y \to Y'$, the (smooth) pushforward \emph{acting on the second  factor $V  \xrightarrow{s} Y$}, \emph{in an unusual way} denoted by ${}_*g:\Cal   M_{m,r}(X,Y)^+ \to \Cal   M_{m+ \op{dim}g,r}(X,Y')^+$, is defined by
$$([X \xleftarrow{p} V  \xrightarrow{s} Y; E])\, {}_*g:= [X \xleftarrow{p} V  \xrightarrow{g \circ s} Y'; E].$$
Here we emphasize that ${}_*g$ is written on the right side of $([X \xleftarrow{p} V  \xrightarrow{s} Y; E])$ \emph{not on the left side}.\footnote{ For pushforward the notation ${}_*g$  and writing it on the right side and for pullback the notation ${}^*g$  and writing it on the right side were suggested by the referee, whom we appreciate for suggesting such an interesting notation. }
(Note that $m = \op{dim} s$ and $\op{dim} (g \circ s) = \op{dim}s + \op{dim}g = m + \op{dim}g$.)
\end{enumerate}
\item (Pullback) 
\begin{enumerate}
\item For a \emph{smooth} map $f:X' \to X$, the (smooth) pullback \emph{acting on the first factor $X \xleftarrow{p} V$}, in a usual way denoted by$f^*:\Cal   M_{m,r}(X,Y)^+ \to \Cal   M_{m+\op{dim}f,r}(X',Y)^+$ is defined by
$$f^*([X \xleftarrow{p} V  \xrightarrow{s} Y; E]) := [X' \xleftarrow{p'} X' \times_X V \xrightarrow{s \circ f'} Y; (f')^*E].$$
Here we consider the following commutative diagram:
$$\xymatrix{
& (f')^*E \ar[d] \ar[r] & E \ar[d] \\
& X' \times_X V \ar[dl]_{p'} \ar[r]^{\qquad f'} & V \ar [dl]^{p} \ar [dr]^{s}\\
X' \ar[r]_f & X &  & Y }
$$
(Note that the left diamond is a fiber square, thus $f':X'\times_X V \to V$ is smooth and $p':X'\times_X V \to X'$ is proper. Note that 
$\op{dim} f' = \op{dim}$ and $\op{dim}(s \circ f') = \op{dim} s + \op{dim}f' = m + \op{dim} f$.)
\item For a \emph{proper} map $g:Y' \to Y$, the (proper) pullback \emph{acting on the second  factor $V  \xrightarrow{s} Y$}, in an unusual way denoted by ${}^*g:\Cal   M_{m,r}(X,Y)^+ \to \Cal   M_{m,r}(X,Y')^+$, is defined by
$$([X \xleftarrow{p} V  \xrightarrow{s} Y; E])\, {}^*g := [X \xleftarrow{p \circ g'} V \times _Y Y' \xrightarrow{s'} Y'; (g')^*E].$$
Here we consider the following commutative diagram:
$$\xymatrix{
& E \ar[d] & (g')^*E \ar[l] \ar[d] \\
& V \ar [dl]^{p} \ar [dr]_{s} & V \times _Y Y' \ar[l]_{g' \quad } \ar[dr]^{s'}\\
X &  & Y & Y' \ar[l]^g}
$$
(Note that the right diamond is a fiber square, thus $s':V\times_Y Y'\to Y'$ is smooth and $g':V\times_Y Y'\to V$ is proper, and $\op{dim} s = \op{dim} s'$.)
\end{enumerate}
\end{enumerate}
\end{defn}
\begin{rem} 
\begin{enumerate}
\item We emphasize that as to pushforward, proper pushforward is concerned with the first factor and smooth pushforward is concerned with the second factor, but as to pullback, the involving factors are exchanged.
\item We remark that when we deal with a smooth map $f$ or $g$, both in pushforward and pullback, the first grading is added by the relative dimension $\op{dim} f$ or $\op{dim} g$, but that when we deal with proper maps, the first grading is not changed. In both pushforward and pullback, the second grading (referring to the dimension of vector bundle) is not changed.
\end{enumerate}
\end{rem}
\begin{pro}\label{proposition1}
The above three operations of product ($\bullet_{\oplus}$ and $\bullet_{\otimes}$), pushforward and pullback satisfy the following properties.
\begin{enumerate}
\item[($A_1$)] {\bf Product is associative}: For three varieties $X,Y,Z, W$ we have
\begin{align*}
(\alp \bullet_{\oplus} \be) \bullet_{\oplus}  \ga & = \alp \bullet_{\oplus} (\be \bullet_{\oplus} \ga) \in \Cal   M_{m+n+\ell,r+k+e}(X,W)^+,\\
(\alp \bullet_{\otimes}\be) \bullet_{\otimes}  \ga & = \alp \bullet_{\otimes} (\be \bullet_{\otimes} \ga) \in \Cal   M_{m+n+\ell,rke}(X,W)^+,
\end{align*}
where $\alp \in \Cal   M_{m,r}(X,Y)^+, \be \in \Cal   M_{n,k}(Y,Z)^+$ and $\be \in \Cal   M_{\ell,e}(Z,W)^+$, 
\item[($A_2$)] {\bf Pushforward is functorial} : 
\begin{enumerate}
\item For two proper maps $f_1:X \to X', f_2:X' \to X''$, we have
$$(f_2 \circ f_1)_* = (f_2)_* \circ (f_1)_*$$
where  $(f_1)_*:\Cal   M_{m,r}(X,Y)^+ \to \Cal   M_{m,r}(X',Y)^+$ and \\
\hspace{0.9cm} $(f_2)_*:\Cal   M_{m,r}(X',Y)^+ \to \Cal   M_{m,r}(X'',Y)^+$.
\item For two smooth maps $g_1:Y \to Y', g_2:Y' \to Y''$ we have
$${}_*(g_2 \circ g_1) = {}_*(g_1) \circ {}_*(g_2)  {}
\, \quad \text{i.e., $\alp\,{}_*(g_2 \circ g_1) = (\alp \, {}_*(g_1)) \, {}_*(g_2)$} $$
where ${}_*(g_1):\Cal   M_{m,r}(X,Y) \to \Cal   M_{m+\op{dim}g_1,r}(X,Y')$ and \\
\hspace{0.9cm}  ${}_*(g_2):\Cal   M_{m+\op{dim}g_1,r}(X,Y') \to \Cal   M_{m+\op{dim}g_1+\op{dim}g_2,r}(X,Y'')$.
\end{enumerate}
\item[($A_2$)'] {\bf Proper pushforward and smooth pushforward commute}: For a proper map $f:X \to X'$ and a smooth map $g:Y \to Y'$ we have
$${}_*g \circ f_*= f_* \circ \, {}_*g, \,\,  \text{i.e.,} \, \, (f_*\alp)\, {}_*g = f_*(\alp\, {}_*g) {} \footnote {The referee remarked that introducing the above unusual notation turns, e.g.,  the condition ($A_2$)' into a ``bimodule-like" condition ($(rm)s=r(ms)$ for an $R$-$S$-bimodule $M$ with two rings $R,S$; $r \in R, s \in M , m \in M$) , which makes the proofs in the paper easier to follow.} \, \,\,  \text {for} \,\,\, \alp \in \Cal M_{m,r}(X,Y)^+$$
i.e, the following diagram commutes:
$$\CD
\Cal   M_{m,r}(X,Y)^+ @> {f_*}>> \Cal   M_{m,r}(X',Y)^+ \\
@V {{}_*g} VV @VV {{}_*g} V\\
\Cal   M_{m+\op{dim}g,r}(X,Y')^+  @>> {f_*} > \Cal   M_{m+\op{dim}g,r}(X',Y')^+ . \endCD
$$ 
\item[($A_3$)] {\bf Pullback is functorial}: 
\begin{enumerate}
\item For two smooth maps $f_1:X \to X', f_2:X' \to X''$ we have
$$(f_2 \circ f_1)^* = (f_1)^* \circ (f_2)^*$$
where  $(f_2)^*:\Cal   M_{m,r}(X'',Y)^+ \to \Cal   M_{m+ \op{dim}f_2,r}(X',Y)^+$ and \\
\hspace{1cm} $(f_1)^*:\Cal   M_{m +\op{dim}f_2,r}(X',Y)^+ \to \Cal   M_{m+  \op{dim}f_2+\op{dim}f_1,r}(X,Y)^+$. 
\item For two proper maps $g_1:Y \to Y', g_2:Y' \to Y''$ we have
$${}^*(g_2 \circ  g_1)  = {}^*(g_2) \circ {}^*(g_1) \, \quad \text{i.e., $\alp\,{}^*(g_2 \circ g_1) = (\alp \, {}^*(g_2)) \, {}^*(g_1)$} $$
where ${}^*(g_1):\Cal   M_{m,r}(X,Y')^+ \to \Cal   M_{m,r}(X,Y)^+$ and \\
\hspace{1cm}  ${}^*(g_2):\Cal   M_{m,r}(X,Y'')^+ \to \Cal   M_{m,r}(X,Y')^+$.
\end{enumerate}
\item[($A_3$)'] {\bf Proper pullback and smooth pullback commute}: For a smooth map $g:X' \to X$ and a proper map $f:Y' \to Y$ we have
$$g^* \circ {}^*f={}^*f \circ g^*,\,\,  \text{i.e.,} \, \, g^*(\alp \, {}^*f) = (g^*\alp)\, {}^*f \, \,\,  \text {for} \,\,\, \alp \in \Cal M_{m,r}(X,Y)^+$$
i.e, the following diagram commutes:
$$\CD
\Cal   M_{m,r}(X,Y)^+ @> {{}^*f}>> \Cal   M_{m,r}(X,Y')^+ \\
@V {g^*} VV @VV {g^*} V\\
\Cal   M_{m+\op{dim}g,r}(X',Y)^+  @>> {{}^*f} > \Cal   M_{m+\op{dim}g,r}(X',Y')^+ . \endCD
$$ 
\item[($A_{12}$)] {\bf Product and pushforward commute}: 
Let $\alp \in \Cal M_{m,r}(X,Y)^+$ and $\be \in \Cal   M_{n,k}(Y,Z)^+$.
 \begin{enumerate} 
\item For a proper morphism $f:X \to X'$, 
\begin{align*}
f_*(\alp \bullet_{\oplus}  \be) & = (f_*\alp) \bullet_{\oplus}  \be   \quad (\in \Cal   M_{m+n,r+k}(X',Z)^+),\\
f_*(\alp \bullet_{\otimes} \be) & = (f_*\alp) \bullet_{\otimes} \be   \quad (\in \Cal   M_{m+n,rk}(X',Z)^+),
\end{align*}
i.e., (for the sake of clarity), the following diagrams commute:
$$\CD
\Cal   M_{m,r}(X,Y)^+ \otimes \Cal   M_{n,k}(Y,Z)^+ @> {\bullet_{\oplus}}>> \Cal   M_{m+n,r+k}(X,Z)^+ \\
@V {f_* \times \op{id}_{\Cal   M_{n,k}(Y,Z)^+}} VV @VV {f}_* V\\
\Cal   M_{m,r}(X',Y)^+ \otimes \Cal   M_{n,k}(Y,Z)^+  @>> {\bullet_{\oplus}} > \Cal   M_{m+n,r+k}(X',Z)^+ . \endCD
$$ 
$$\CD
\Cal   M_{m,r}(X,Y)^+ \otimes \Cal   M_{n,k}(Y,Z)^+ @> {\bullet_{\otimes}}>> \Cal   M_{m+n,rk}(X,Z)^+ \\
@V {f_* \times \op{id}_{\Cal   M_{n,k}(Y,Z)^+}} VV @VV {f}_* V\\
\Cal   M_{m,r}(X',Y)^+ \otimes \Cal   M_{n,k}(Y,Z)^+  @>> {\bullet_{\otimes}} > \Cal   M_{m+n,rk}(X',Z)^+ . \endCD
$$
\item For a smooth morphism $g:Z \to Z'$, 
\begin{align*}
(\alp \bullet_{\oplus}  \be)\, {}_*g & = \alp \bullet_{\oplus}  \be\, {}_*g   \quad (\in \Cal   M_{m+n+\op{dim}g,r+k}(X,Z')^+),\\
(\alp \bullet_{\otimes} \be)\, {}_*g & = \alp \bullet_{\otimes} \be\, {}_*g   \quad (\in \Cal   M_{m+n+\op{dim}g,rk}(X,Z')^+),
\end{align*}
i.e., as to the product $\bullet_{\oplus}$ (the case of $\bullet_{\otimes}$ is omitted), the following diagram commutes:
$$\CD
\Cal   M_{m,r}(X,Y)^+ \otimes \Cal   M_{n,k}(Y,Z)^+ @> {\bullet_{\oplus}}>> \Cal   M_{m+n,r+k}(X,Z)^+ \\
@V {\op{id}_{\Cal   M_{m,r}(X,Y)^+} \otimes \, \, {}_*g} VV @VV {{}_*g} V\\
\quad \quad \Cal   M_{m,r}(X,Y)^+ \otimes \Cal   M_{n+\op{dim}g,k}(Y,Z')^+  @>> {\bullet_{\oplus}} > \Cal   M_{m+n+\op{dim}g,r+k}(X,Z')^+ . \endCD
$$ 
\end{enumerate}
\item[($A_{13}$)] {\bf Product and pullback commute}: Let $\alp \in \Cal   M_{m,r}(X,Y)^+$ and 

\noindent
$\be \in    \Cal   M_{n,k}(Y,Z)^+$.
 
\begin{enumerate} 
\item For a smooth morphism $f:X' \to X$,
\begin{align*} 
f^*(\alp \bullet_{\oplus}  \be) & = (f^*\alp) \bullet_{\oplus}  \be \quad ( \in \Cal   M_{m+n+ \op{dim}f,r+k}(X',Z)^+ ),\\
f^*(\alp \bullet_{\otimes}  \be) & = (f^*\alp) \bullet_{\otimes} \be \quad ( \in \Cal   M_{m+n+ \op{dim}f,rk}(X',Z)^+ ),
\end{align*}
i.e., as to the product $\bullet_{\oplus}$, the following diagram commutes:
$$\CD
\Cal   M_{m,r}(X,Y)^+ \otimes \Cal   M_{n,k}(Y,Z)^+ @> {\bullet_{\oplus}}>> \Cal   M_{m+n,r+k}(X,Z)^+ \\
@V {f^* \times \op{id}_{\Cal   M_{n,k}(Y,Z)^+}} VV @VV {f^*} V\\
\Cal   M_{m+\op{dim}f,r}(X',Y)^+ \otimes \Cal   M_{n,k}(Y,Z)^+ @>> {\bullet_{\oplus}} > \Cal   M_{m+n+\op{dim},r+k}(X',Z)^+   . \endCD
$$ 

\item For a proper morphism $g:Z' \to Z$, 
\begin{align*} 
(\alp \bullet_{\oplus} \be)\, {}^*g & = \alp \bullet_{\oplus} \be \, {}^*g \quad ( \in \Cal   M_{m+n,r+k}(X,Z')^+ ),\\
(\alp \bullet_{\otimes} \be)\, {}^*g & = \alp \bullet_{\otimes} \be\, {}^*g \quad ( \in \Cal   M_{m+n,rk}(X,Z')^+ ),
\end{align*}
i.e., as to the product $\bullet_{\oplus}$, the following diagram commutes:
$$\CD
\Cal   M_{m,r}(X,Y)^+ \otimes \Cal   M_{n,k}(Y,Z)^+ @> {\bullet_{\oplus}}>> \Cal   M_{m+n,r+k}(X,Z)^+ \\
@V {\op{id}_{\Cal   M_{m,r}(X,Y)^+} \otimes \, \, {}^*g} VV @VV {{}^*g} V\\
\quad \quad \Cal   M_{m,r}(X,Y)^+ \otimes \Cal   M_{n,k}(Y,Z')^+  @>> {\bullet_{\oplus}} > \Cal   M_{m+n,r+k}(X,Z')^+ . \endCD
$$ 
\end{enumerate}
\item[($A_{23}$)] {\bf Pushforward and pullback commute}: For $\alp \in \Cal   M_{m,r}(X,Y)^+$
 
\begin{enumerate} 
\item (proper pushforward and proper pullback commute) For proper morphisms $f:X \to X'$ and $g:Y' \to Y$ and for $\alp \in \Cal   M_{m,r}(X,Y)^+$
$$(f_*\alp)\, {}^*g = f_*(\alp\, {}^*g) \quad ( \in \Cal   M_{m,r}(X',Y')^+ ),$$
i.e., the following diagram commutes:
$$\CD
\Cal   M_{m,r}(X,Y)^+ @> {f_*}>> \Cal   M_{m,r}(X',Y)^+ \\
@V {{}^*g} VV @VV {{}^*g} V\\
M_{m,r}(X,Y')^+  @>> {f_*} > \Cal   M_{m,r}(X',Y')^+ . \endCD
$$ 
\item (smooth pushforward and smooth pullback commute) For smooth morphisms $f:X' \to X$ and $g:Y \to Y'$ and for $\alp \in \Cal   M_{m,r}(X,Y)^+$
$$f^*(\alp\, {}_*g) = (f^*\alp)\, {}_*g \quad ( \in \Cal   M_{m+ \op{dim} f + \op{dim}g,r}(X',Y')^+ ),$$
i.e., the following diagram commutes:
$$\CD
\Cal   M_{m,r}(X,Y)^+ @> {{}_*g}>> \Cal   M_{m+\op{dim}g,r}(X,Y')^+ \\
@V {f^*} VV @VV {f^*} V\\
M_{m+\op{dim}f,r}(X',Y)^+  @>> {{}_*g} > \Cal   M_{m+\op{dim}f+\op{dim}g,r}(X',Y')^+ . \endCD
$$
\item (proper pushforward and smooth pullback ``commute" in the following sense) For the following fiber square
 $$\CD
\widetilde X @> {\widetilde f}>> X''\\
@V {\widetilde g} VV @VV {g} V\\
X' @>> {f} > X \endCD
$$ 
with $f$ proper and $g$ smooth, we have
$$g^*f_* = \widetilde f_* \widetilde g^*,$$
i.e., the following diagram commutes:
$$\CD
\Cal   M_{m,r}(X',Y)^+ @> {f_*}>> \Cal   M_{m,r}(X,Y)^+ \\
@V {\widetilde g^*} VV @VV {g^*} V\\
M_{m+\op{dim}g,r}(\widetilde X,Y)^+  @>> {\widetilde f_*} > \Cal   M_{m+\op{dim}g,r}(X'',Y)^+ . \endCD
$$
(Note that $\op{dim}\widetilde g = \op{dim}g$.)

\item (smooth pushforward and proper pullback ``commute" in the following sense) For the following fiber square
 $$\CD
\widetilde Y @> {\widetilde f}>> Y''\\
@V {\widetilde g} VV @VV {g} V\\
Y' @>> {f} > Y \endCD
$$ 
with $f$ proper and $g$ smooth, for $\alp \in \Cal   M_{m,r}(X,Y'')^+$ we have
$$(\alp\, {}_*g)\, {}^*f = (\alp\, {}^*\widetilde f)\, {}_*\widetilde g,$$
i.e., the following diagram commutes:
$$\CD
\Cal   M_{m,r}(X,Y'')^+ @> {{}_*g}>> \Cal   M_{m+\op{dim}g,r}(X,Y)^+ \\
@V {{}^*\widetilde f} VV @VV {{}^*f} V\\
M_{m,r}(X,\widetilde Y)^+  @>> {{}_*\widetilde g} > \Cal   M_{m+\op{dim}g,r}(X,Y')^+ . \endCD
$$
(Note that $\op{dim}\widetilde g = \op{dim}g$.)

\end{enumerate}
\item[($A_{123}$)] {\bf ``Projection formula"}: 
\begin{enumerate}
\item For a smooth morphism $g:Y \to Y'$ and $\alp \in \Cal   M_{m,r}(X,Y)^+$ and $\be \in \Cal   M_{n,k}(Y', Z)^+$,
\begin{align*} 
(\alp\, {}_*g) \bullet_{\oplus} \be & = \alp \bullet_{\oplus} g^*\be \quad ( \in \Cal   M_{m+n+ \op{dim}g,r+k}(X,Z)^+ ),\\
(\alp\, {}_*g) \bullet_{\otimes} \be & = \alp \bullet_{\otimes} g^*\be \quad ( \in \Cal   M_{m+n+ \op{dim}g,rk}(X,Z)^+ ),
\end{align*}
i.e., as to the product $\bullet_{\oplus}$, the following diagram commutes:
$$
\xymatrix{
\Cal   M_{m,r}(X,Y)^+ \otimes \Cal   M_{n,k}(Y',Z)^+ \ar[d]_{{}_*g \otimes \op{id}_{\Cal   M_{n,k}(Y',Z)^+}} \ar[rr]^{\op{id}_{\Cal   M_{m,r}(X,Y)^+} \times  g^*} && \Cal   M_{m,r}(X,Y)^+ \otimes \Cal   M_{n+\op{dim}g,k}(Y,Z)^+ \ar[d]^{\bullet_{\oplus}}  \\
\Cal   M_{m+\op{dim}g,r}(X,Y')^+ \otimes \Cal   M_{n,k}(Y',Z)^+ \ar[rr]_{\bullet_{\oplus}} && \Cal   M_{m+n+\op{dim}g,r+k}(X,Z)^+   .
}
$$ 
\item For a proper map $g:Y' \to Y$, $\alp \in \Cal   M_{m,r}(X,Y)^+$ and $\be \in \Cal   M_{n,k}(Y', Z)^+$,
\begin{align*} 
(\alp\, {}^*g) \bullet_{\oplus} \be & = \alp \bullet_{\oplus} g_*\be \quad ( \in \Cal   M_{m+n,r+k}(X,Z)^+ ),\\
(\alp\, {}^*g) \bullet_{\otimes}\be & = \alp \bullet_{\otimes} g_*\be \quad ( \in \Cal   M_{m+n,rk}(X,Z)^+ ),
\end{align*}
i.e., as to the product $\bullet_{\oplus}$, the following diagram commutes:
$$\xymatrix{
\Cal   M_{m,r}(X,Y)^+ \otimes \Cal   M_{n,k}(Y',Z)^+ \quad  \ar[d]_{{}^*g \otimes \op{id}_{\Cal   M_{n,k}(Y',Z)^+}} \ar[rr]^{\op{id}_{\Cal   M_{m,r}(X,Y)^+} \times  g_*} && \Cal   M_{m,r}(X,Y)^+ \otimes \Cal   M_{n,k}(Y,Z)^+  \ar[d]^{\bullet_{\oplus}}\\
\Cal   M_{m,r}(X,Y')^+ \otimes \Cal   M_{n,k}(Y',Z)^+ \ar[rr]_{\bullet_{\oplus}} && \Cal M_{m+n,r+k}(X,Z)^+   . 
}
$$ 

\end{enumerate}
\end{enumerate}
\end{pro}

\begin{proof} It is straightforward. For the sake of readers' convenience. we give proofs to (c) and (d) of $A_{23}$ and the ``projection formula" with respect to the product $\bullet_{\oplus}$.

(c) Let $[X' \xleftarrow p V \xrightarrow s Y; E] \in \Cal   M_{m, r}(X',Y)^+$ and consider the following commutative diagrams:
$$
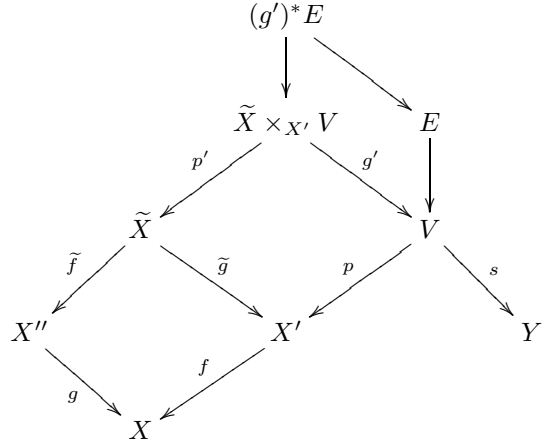
\xymatrix{
& &  (g')^*E \ar[d] \ar[dr]\\
& & \widetilde X \times_{X'} V \ar[dl]_{p'} \ar[dr]^{g'} & E \ar[d] \\
& \widetilde X \ar[dl]_{\widetilde f} \ar[dr]^{\widetilde g} & & V \ar[dl]_p \ar[dr]^s\\
X'' \ar[dr]_g & & X' \ar[dl]_f & & Y\\
& X 
}
$$
\begin{align*}
\widetilde f_*\widetilde g^*([X' \xleftarrow p V \xrightarrow s Y; E]) & = \widetilde f_* ([\widetilde X \xleftarrow {p'} \widetilde X \times_{X'} V \xrightarrow {s \circ g'} Y; (g')^*E] \\
& = [X'' \xleftarrow {\widetilde f \circ p'} \widetilde X\times_{X'} V \xrightarrow {s \circ g'} Z; (g')^*E] \\
& = g^*([X \xleftarrow {f \circ p} V \xrightarrow s Z; E])\\
& = g^* (f_*([X' \xleftarrow p V \xrightarrow s Y; E])\\
& = g^*f_*([X' \xleftarrow p V \xrightarrow sY; E]).
\end{align*}
Hence we have that $g^*f_* = \widetilde f_*\widetilde g^*.$ \\

(d) Let $[X \xleftarrow p V \xrightarrow s Y''; E] \in \Cal   M_{m, r}(X,Y'')^+$ and consider the following commutative diagrams:
$$
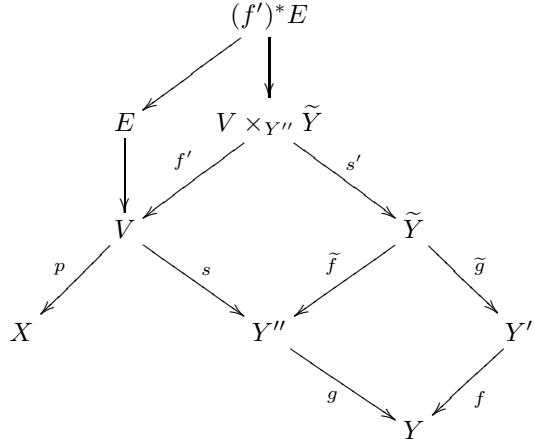
\xymatrix{
& &  (f')^*E \ar[d] \ar[dl]\\
& E \ar[d]  & V \times _{Y''} \widetilde Y  \ar[dl]_{f'} \ar[dr]^{s'}\\
& V \ar[dl]_{p} \ar[dr]^{s} & & \widetilde Y \ar[dl]_{\widetilde f} \ar[dr]^{\widetilde g}\\
X  & & Y'' \ar[dr]_g & & Y' \ar[dl]^f \\
& & &  Y 
}
$$
\begin{align*}
([X \xleftarrow p V \xrightarrow s Y''; E]_*g)^*f & = [X \xleftarrow {p} V \xrightarrow {g \circ s} Y;E]^*f \\
& = [X \xleftarrow {p \circ f'} V \times _{Y''} \widetilde Y  \xrightarrow {\widetilde g \circ s'} Y'; (f')^*E] \\
& = [X \xleftarrow {p \circ f'} V \times _{Y''} \widetilde Y   \xrightarrow {s'} \widetilde Y; (f')^*E]_*\widetilde g\\
& = ([X \xleftarrow p V \xrightarrow s Y''; E]^*\widetilde f)_*\widetilde g.
\end{align*}
Hence we have that $(\alp\, {}_*g)\, {}^*f = (\alp\, {}^*\widetilde f)\, {}_*\widetilde g.$\\

($A_{123}$): Let $\alp = [X \xleftarrow p V \xrightarrow s Y;E], \be =[Y' \xleftarrow q W \xrightarrow t Z; F].$\\

(a) For the first projection formula, we consider the following commutative diagram.\\
$$\xymatrix{
& & (q'')^*E \oplus (g'\circ s')^*F \ar[d] \\
& E \ar[d] & V\times_Y (Y'\times_{Y'} W) \ar[dl]_{q''} \ar[dr]^{s'} & (g')^*F \ar[d] \ar[dr] &&\\
& V \ar[dl]_{p} \ar[dr]^{s} & &  Y \times_{Y'} W \ar[dl]_{q'}  \ar[dr]^{g'} & F \ar[d] \\
X && Y \ar[dr]_g  && W \ar[dl]_q \ar[dr]^t  && \\
& & & Y' & & Z
}
$$
Since $\alp\, {}_*g =[X \xleftarrow p V \xrightarrow {g \circ s} Y'; E]$, we have
\begin{align*}
\alp\, {}_*g  \bullet_{\oplus} \be & = [X \xleftarrow {p \circ q'} V\times_{Y'} W \xrightarrow {t \circ (g' \circ s')} Z; (q'')^*E\oplus(g'\circ s')^*F]\\
& = [X \xleftarrow {p \circ q'} V\times_Y (Y'\times_{Y'} W) \xrightarrow {(t \circ g') \circ s'} Z; (q'')^*E\oplus (s')^*((g')^*F)]\\
& = [X \xleftarrow p V \xrightarrow s Y;E] \bullet_{\oplus} [Y \xleftarrow {q'} Y'\times_{Y'} W \xrightarrow {t \circ g'} Z;(g')^*F]\\
&= \alp \bullet_{\oplus} g^*\be.
\end{align*}

(b) For the second projection formula, we consider the following commutative diagram:
$$\xymatrix{
& & & (q')^*(g')^*E \oplus (s'')^*F\ar[d] \\
& & (g')^*E \ar[d] \ar[dl]& (V\times_Y Y') \times_{Y'} W \ar[dl]_{q'} \ar[dr]^{s''} & F \ar[d] \\
& E \ar[d] & V \times_Y Y' \ar[dl]_{g'} \ar[dr]^{s'} & &  W \ar[dl]_q  \ar[dr]^t && \\
&  V \ar[dl]_p \ar[dr]^s & & Y' \ar[dl]_g & & Z\\
X & & Y
}
$$
Then, since $\alp\, {}^*g = [X \xleftarrow {p \circ g'} V \times_Y Y' \xrightarrow {s'} Y'; (g')^*E]$, we have
\begin{align*}
\alp\, {}^*g  \bullet_{\oplus} \be &= [X \xleftarrow {p \circ g' \circ q'}(V\times_Y Y') \times_{Y'} W \xrightarrow {t \circ s''} Z; (q')^*(g')^*E \oplus(s'')^*F]\\
\end{align*}
\begin{align*}
& = [X \xleftarrow {p \circ (g' \circ q')} V\times_Y W \xrightarrow {t \circ s''} Z; (g' \circ q')^*E \oplus(s'')^*F]\\
& = [X \xleftarrow p V \xrightarrow Y; E] \bullet_{\oplus} [Y \xleftarrow {g \circ q} W \xrightarrow t  Z; F]\\
& = \alp \bullet_{\oplus}  g_*\be.
\end{align*}
\end{proof}
\begin{rem} $1_X=[X \xleftarrow {\op{id}_X} X \xrightarrow {\op{id}_X} X] \in \Cal   M_{0, 0}(X,X)^+$ satisfies that $1_X \bullet_{\oplus} \alp = \alp$ for any element $\alp \in \Cal   M_{m, r}(X, Y)^+$ and $\be \bullet_{\oplus} 1_X = \be$ for any element $\be \in \Cal   M_{m, r}(Y, X)^+$. As to the product $\bullet_{\otimes}$, let $\jeden_X=[X \xleftarrow {\op{id}_X} X \xrightarrow {\op{id}_X} X; \Cal  O_X] \in \Cal   M_{0, 0}(X,X)^+$ (where $\Cal  O_X$ is the trivial line bundle) satisfies that $\jeden_X \bullet_{\otimes } \alp = \alp$ for any element $\alp \in \Cal   M_{m, r}(X, Y)^+$ and $\be \bullet_{\otimes} \jeden_X = \be$ for any element $\be \in \Cal   M_{m, r}(Y, X)^+$.
\end{rem} 
\begin{rem} $\Cal   M_{m, r}(X, Y)^+$ with the product $\bullet_{\oplus}$ considered shall be denoted by $\Cal   M_{m, r}(X, Y)^+_{\oplus}$ and $\Cal   M_{m, r}(X, Y)^+$ with the product $\bullet _{\otimes}$ shall be denoted by $\Cal   M_{m, r}(X, Y)^+_{\otimes}$.
\end{rem}
\begin{rem} In a way similar to that of constructing $\mathbb {OM}^{prop}_{sm}(X \xrightarrow f Y)$ generated by the isomorphism classes $[V \xrightarrow h X; L_1, \cdots, L_r]$ such that $h$ is proper and $f \circ h$ is smooth and $L_i$'s are line bundles over $V$, we can replace $[V \xrightarrow h X; L_1, \cdots, L_r]$ by the isomorphism classes $[V \xrightarrow h X;E]$ with line bundles $\{L_1, \cdots, L_r\}$ being replaced by one vector bundle $E$. Then in \cite{AY} (cf.\cite{An2, An3}) T. Annala and the author have generalized Lee--Pandharipande's algebraic cobordism $\omega_{*,*}(X)$ of vector bundles  \cite{LeeP}(also see \cite{LP}) to a bivariant-theoretic analogue $\Omega^{*,*}(X \to Y)$ in a way similar to that of Annala's construction of the bivariant derived algebraic cobordism $\Omega^*(X \to Y)$.

\end{rem}
\section{Cobordism bicycles of finite tuples of line bundles}
Here we consider a ``bi-variant" analogue of Levine--Morel's construction of algebraic cobordism $\Omega_*(X)$ \cite{LM} via cobordism bicycles of finite tuples of line bundles. 

In \cite{LM} Levine and Morel consider what they call a cobordism cycle
$$[M \xrightarrow {p} X; L_1, \cdots L_r],$$
which is the isomorphism class of $(M \xrightarrow {p} X; L_1, \cdots L_r)$ where $M$ is smooth and irreducible, $p:M \to X$ is proper and $L_i$'s are line bundles over $M$. $(M \xrightarrow {p} X; L_1, \cdots L_r)$ and $(M' \xrightarrow {p'} X; L'_1, \cdots L'_r)$ are called isomorphic if there exists an isomorphism $h:M \to M'$ such that (1) the following diagram commutes
$$\xymatrix{
M \ar[dd]_h^{\cong} \ar[dr]^p\\
& X\\
M' \ar[ur]_{p'}
}
$$
and (2) there exists a bijection $\sigma:\{1, 2, \cdots, r\} \cong \{1, 2, \cdots, r\}$ such that $L_i \cong h^*L_{\sigma(i)}$.

So, here we consider a ``bicycle" version of this cobordism cycle. Namely we consider the isomorphism class of
a cobordism bicycle of a finite tuple of line bundles, instead of one vector bundle. 

\begin{defn} $(X \xleftarrow{p} V  \xrightarrow{s} Y; L_1, \cdots, L_r)$ and $(X \xleftarrow{p'} V'  \xrightarrow{s'} Y; L'_1, \cdots, L'_r)$ are called isomorphic if the following conditions hold:
\begin{enumerate}
\item There exists an isomorphism $h:V \cong V'$ such $(X \xleftarrow{p} V  \xrightarrow{s} Y) \cong (X \xleftarrow{p'} V'  \xrightarrow{s'} Y)$ as correspondences (as in (1) of Definition \ref{bicycle}), 
\item There exists a bijection $\sigma:\{1, 2, \cdots, r\} \cong \{1, 2, \cdots, r\}$ such that $L_i \cong h^*L_{\sigma(i)}$.
\end{enumerate}
The isomorphism class of $(X \xleftarrow{p} V  \xrightarrow{s} Y; L_1, \cdots, L_r)$ is called simply \emph{a cobordism bicycle} (instead of a cobordism bicycle of line bundles) and denoted by $[X \xleftarrow{p} V  \xrightarrow{s} Y; L_1, \cdots, L_r]$.
\end{defn}
It is clear that when $Y=pt$ is a point the cobordism bicycle $[X \xleftarrow{p} V  \xrightarrow{s} pt; L_1, \cdots, L_r]$
is the same as the cobordism cycle $[V \xrightarrow{p} X; L_1, \cdots, L_r]$.
\begin{defn}\label{def1} We define 
$$\Cal {Z}^i(X, Y)$$
to be the free abelian group generated by the set of isomorphism classes of cobordism bicycle 
$$[X \xleftarrow{p} V  \xrightarrow{s} Y; L_1, L_2, \cdots, L_r],$$
such that $-i+r = \op{dim}s = \op{dim} V - \op{dim}Y$, modulo the following additive relation
\begin{align*}
[X \xleftarrow {p_1} V_1 \xrightarrow {s_1} Y; L_1, \cdots, L_r] & + [X \xleftarrow {p_2} V_2 \xrightarrow {s_2} Y; L'_1, \cdots, L'_r] \\
& := [X \xleftarrow {p_1 \sqcup p_2 } V_2 \xrightarrow {s_1 \sqcup s_2} Y; L_1 \sqcup L'_1, \cdots, L_r \sqcup L'_r].
\end{align*}
\end{defn}
\begin{rem} Such a grading is due to the requirement that for $Y =pt$ we want to have $\Cal {Z}^i(X, pt) = \Cal Z_{-i}(X)$. Here we note that $V$ is smooth, since $s:V \to pt$ is smooth.
According to the definition (\cite[Definition 2.1.6]{LM}) of grading of Levine-Morel's algebraic pre-cobordism $\Cal Z_*(X)$, the degree (or dimension) of the cobordism cycle $[V \xrightarrow {p}  X; L_1, \cdots, L_r] \in \Cal Z_*(X)$ is $\op{dim} V - r$, i.e.

$[V \xrightarrow {p}  X; L_1, \cdots, L_r] \in \Cal Z_{-i}(X) \Longleftrightarrow -i = \op{dim} V - r$, namely, $-i + r = \op{dim} V$. 
\end{rem}

The following definitions are similar to those in the case of cobordism bicycles of vector bundles, but we write down them for the sake of convenience.
\begin{defn}\label{ppp of cbc}
\begin{enumerate} 
\item (Product of cobordism bicycles) We define the product of cobordism bicycles as follows:
$$\bullet : \Cal {Z}^i(X, Y) \otimes \Cal {Z}^j(Y,Z) \to \Cal {Z}^{i+j}(X, Z)$$
\begin{align*}
& [X \xleftarrow{p_1} V_1  \xrightarrow{s_1} Y; L_1, L_2, \cdots, L_r] \bullet [Y \xleftarrow{p_2} V_2  \xrightarrow{s_2} Z; M_1, M_2, \cdots, M_k] \\
& := [(X \xleftarrow{p_1} V_1  \xrightarrow{s_1} Y) \circ (Y \xleftarrow{p_2} V_2  \xrightarrow{s_2} Z) ; \widetilde{p_2}^*L_1, \cdots, \widetilde{p_2}^*L_r, \widetilde{s_1}^*M_1, \cdots \widetilde{s_1}^*M_k]
\end{align*}

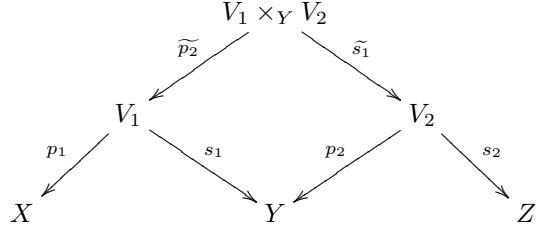
$$\xymatrix{
&& V_1\times_Y V_2 \ar [dl]_{\widetilde{p_2}} \ar[dr]^{\widetilde{s_1}} &&\\
& V_1 \ar [dl]_{p_1} \ar [dr]^{s_1} && V_2 \ar [dl]_{p_2} \ar[dr]^{s_2}\\
X & &  Y && Z }
$$
\item (Proper pushforward and smooth pushforward of cobordism bicycles) 
\begin{enumerate}
\item For a \emph{proper} map $f:X \to X'$, the (proper) pushforward (with respect to the first factor) $f_*:\Cal  Z^i(X,Y) \to \Cal  Z^i(X',Y)$ is defined by
$$f_*([X \xleftarrow{p} V  \xrightarrow{s} Y; L_1, L_2, \cdots, L_r]):= [X' \xleftarrow{f \circ p} V  \xrightarrow{s} Y; L_1, L_2, \cdots, L_r].$$
\item For a \emph{smooth} map $g:Y \to Y'$, the (smooth) pushforward (with respect to the second factor) ${}_*g:\Cal  Z^i(X,Y) \to \Cal  Z^{i+ \op{dim}g}(X,Y')$ is defined by
$$([X \xleftarrow{p} V  \xrightarrow{s} Y; L_1, L_2, \cdots, L_r]) \, {}_*g:= [X \xleftarrow{p} V  \xrightarrow{g \circ s} Y'; L_1, L_2, \cdots, L_r].$$
\end{enumerate}
\item (Smooth pullback and proper pullback  of cobordism bicycles) 
\begin{enumerate}
\item For a \emph{smooth} map $f:X' \to X$, the (smooth) pullback (with respect to the first factor) $f^*:\Cal  Z^i(X,Y) \to \Cal  Z^{i+\op{dim}f}(X',Y)$ is defined by
\begin{align*}
f^*([X \xleftarrow{p} V & \xrightarrow{s} Y; L_1, L_2, \cdots, L_r]) \\
& := [X' \xleftarrow{p'} X' \times_X V \xrightarrow{s \circ f'} Y; (f')^*L_1, (f')^*L_2, \cdots, (f')^*L_r].
\end{align*}
Here we consider the following commutative diagram:
$$
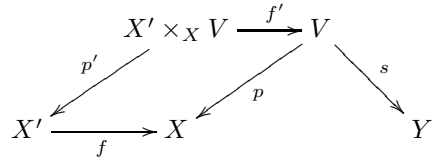
\xymatrix{
& X' \times_X V \ar[dl]_{p'} \ar[r]^{\qquad f'} & V \ar [dl]^{p} \ar [dr]^{s}\\
X' \ar[r]_f & X &  & Y }
$$
\item For a \emph{proper} map $g:Y' \to Y$, the (proper) pullback (with respect to the second factor) ${}^*g:\Cal  Z^i(X,Y) \to \Cal  Z^i(X,Y')$ is defined by
\begin{align*}
([X \xleftarrow{p} V   & \xrightarrow{s} Y;  L_1, L_2, \cdots, L_r])\, {}^*g\\
& := [X \xleftarrow{p \circ g'} V \times _Y Y' \xrightarrow{s'} Y'; (g')^*L_1, (g')^*L_2, \cdots, (g')^*L_r].
\end{align*}
Here we consider the following commutative diagram:
$$\xymatrix{
& V \ar [dl]^{p} \ar [dr]_{s} & V \times _Y Y' \ar[l]_{g' \quad } \ar[dr]^{s'}\\
X &  & Y & Y' \ar[l]^g}
$$
\end{enumerate}

\end{enumerate}
\end{defn}
\begin{rem} As in the case of cobordism bicycles of vector bundles, the involvement of factors are exchanged in pushforward and pullback for proper morphisms and smooth morphisms, and also, in both pushforward and pullback, as long as smooth morphisms are involved, the grading is added by the relative dimension of the smooth morphism.
\end{rem}
\begin{rem} In the case of cobordism bicycles of vector bundles, we have two kind of products by Whitney sum and tensor product. However, in the present case  of cobordism bicycles, the product is a kind of ``Whitney sum'', or a ``mock'' Whitney sum, i.e., the sum of two finite sets of line bundles.
\end{rem}

Clearly we have the following proposition, as in the case of cobordism bicycles of vector bundles in \S 4. 
\begin{pro}\label{pro-bi-v}
The above three operations satisfy the following properties.
\begin{enumerate}
\item[($A_1$)] {\bf Product is associative}: For three varieties $X,Y,Z, W$ and $\alp \in \Cal   Z^i(X,Y)$, 
$\be \in \Cal   Z^j(Y,Z), \ga \in \Cal   Z^k(Z,W)$, we have
$$(\alp \bullet  \be) \bullet  \ga = \alp \bullet  (\be \bullet  \ga) \quad (\in \Cal   Z^{i+j+k}(X,W) )$$
\item[($A_2$)] {\bf Pushforward is functorial} : 
\begin{enumerate}
\item For two proper maps $f_1:X \to X', f_2:X' \to X''$ and 

$(f_1)_*:Z^i(X,Y) \to \Cal   Z^i(X',Y)$, 

$(f_2)_*:Z^i(X',Y) \to \Cal   Z^i(X'',Y)$, we have
$$(f_2 \circ f_1)_* = (f_2)_* \circ (f_1)_*.$$
\item For two smooth maps $g_1:Y \to Y', g_2:Y' \to Y''$ and 

${}_*(g_1):\Cal   Z^i(X,Y) \to \Cal   Z^{i+\op{dim}g_1}(X,Y')$,

${}_*(g_2):Z^{i+\op{dim}g_1} (X,Y') \to \Cal   Z^{i+\op{dim}g_1 + \op{dim}g_2}(X,Y'')$, we have
$${}_*(g_2 \circ g_1) = {}_*(g_1) \circ {}_*(g_2), \, \text{i.e., $\alp \, {}_*(g_2 \circ g_1) = (\alp \, {}_*(g_1) ) \, {}_*(g_2)$} $$
\end{enumerate}
\item[($A_2$)'] {\bf Proper pushforward and smooth pushforward commute}: For a proper map $f:X \to X'$ and a smooth map $g:Y \to Y'$ we have
$${}_*g \circ f_* = f_* \circ \, {}_*g,\,\,  \text{i.e.,} \, \, (f_*\alp)\, {}_*g = f_*(\alp\, {}_*g), \, \text{simply denoted by $f_* \alp\, {}_*g$} $$
i.e, the following diagram commutes:
$$\CD
\Cal   Z^i(X,Y) @> {f_*}>> \Cal   Z^i(X',Y)\\
@V {{}_*g} VV @VV {{}_*g} V\\
\Cal   Z^{i+\op{dim}g}(X,Y')  @>> {f_*} > \Cal   Z^{i+\op{dim}g}(X',Y') . \endCD
$$ 
\item[($A_3$)] {\bf Pullback is functorial}: 
\begin{enumerate}
\item For two smooth maps $f_1:X \to X', f_2:X' \to X''$ and 

$(f_1)^*:\Cal   Z^{i+\op{dim}f_2}(X',Y) \to \Cal   Z^{i+\op{dim}f_2 + \op{dim}f_1}(X,Y)$,

$(f_2)^*:\Cal   Z^i(X'',Y) \to \Cal   Z^{i+\op{dim}f_2}(X',Y)$, we have 
$$(f_2 \circ f_1)^* = (f_1)^* \circ (f_2)^*.$$
\item For two proper maps $g_1:Y \to Y', g_2:Y' \to Y''$ and 

${}^*(g_1):\Cal   Z^i(X,Y') \to \Cal   Z^i(X,Y)$,

${}^*(g_2):\Cal   Z^i(X,Y'') \to \Cal   Z^i(X,Y')$, we have
$${}^*(g_2 \circ g_1) = {}^*(g_2) \circ {}^*(g_1), \, \text{i.e., $\alp \, {}^*(g_2 \circ g_1) = (\alp \, {}^*(g_2) ) \, {}_*(g_1)$}$$
\end{enumerate}
\item[($A_3$)'] {\bf Proper pullback and smooth pullback commute}: For a smooth map $g:X' \to X$ and a proper map $f:Y' \to Y$ we have
$$g^*\circ \, {}^*f = {}^*f \circ g^*,\,\,  \text{i.e.,} \, \, g^*(\alp \, {}^*f) = (g^*\alp)\, {}^*f, \, \text{simply denoted by $g^* \alp\, {}^*f$} $$
i.e, the following diagram commutes:
i.e, the following diagram commutes:
$$\CD
\Cal   Z^i(X,Y) @> {{}^*f}>> \Cal   Z^i(X,Y')\\
@V {g^*} VV @VV {g^*} V\\
\Cal   Z^{i+\op{dim}g}(X',Y)  @>> {{}^*f} > \Cal   Z^{i+\op{dim}g}(X',Y') . \endCD
$$ 
\item[($A_{12}$)] {\bf Product and pushforward commute}: For $\alp \in \Cal   Z^i(X,Y)$ and $\be \in \Cal   Z^j(Y,Z)$.
\begin{enumerate} 
\item For a proper morphism $f:X \to X'$, 
$$f_*(\alp \bullet \be) = (f_*\alp) \bullet \be \quad (\in \Cal   Z^{i+j}(X',Z)),$$
i.e., (for the sake of clarity we write down) the following diagram commutes:
$$\CD
\Cal Z^i(X,Y) \otimes \Cal  Z^j(Y,Z) @> {\bullet}>> \Cal  Z^{i+j}(X,Z)\\
@V {f_* \otimes \op{id}_{\Cal  Z^j(Y,Z)}} VV @VV {f}_* V\\
\Cal  Z^i(X',Y) \otimes \Cal  Z^j(Y,Z) @>> {\bullet} > \Cal  Z^{i+j}(X',Z) . \endCD
$$ 
\item For a smooth morphism $g:Z \to Z'$, 
$$(\alp \bullet \be)\, {}_*g = \alp \bullet \be \, {}_*g \quad (\in \Cal   Z^{i+j+\op{dim}g}(X,Z')),$$
i.e., the following diagram commutes:
$$\CD
\Cal  Z^i(X,Y)\otimes \Cal  Z^j(Y,Z) @> {\bullet}>> \Cal   Z^{i+j}(X,Z) \\
@V {\op{id}_{\Cal  Z^i(X,Y)} \otimes \,\, {}_*g} VV @VV {{}_*g} V\\
\Cal Z^i(X,Y)\otimes \Cal  Z^{j+\op{dim}g}(Y,Z') @>> {\bullet} > \Cal   Z^{i+j+\op{dim}g}(X,Z') . \endCD
$$ 
\end{enumerate}
\item[($A_{13}$)] {\bf Product and pullback commute}: For  $\alp \in \Cal   Z^i(X,Y)$ and $\be \in \Cal   Z^j(Y,Z)$.
 \begin{enumerate} 
\item For a smooth morphism $f:X' \to X$, 
$$f^*(\alp \bullet \be) = (f^*\alp) \bullet \be \quad (\in \Cal   Z^{i+j+\op{dim}f}(X',Z)),$$
i.e., the following diagram commutes:
$$\CD
\Cal Z^i(X,Y) \otimes \Cal  Z^j(Y,Z) @> {\bullet}>> \Cal  Z^{i+j}(X,Z)\\
@V {f^* \otimes \op{id}_{\Cal  Z^j(Y,Z)}} VV @VV {f^*} V\\
\Cal  Z^{i+\op{dim}f}(X',Y) \otimes \Cal  Z^j(Y,Z) @>> {\bullet} > \Cal  Z^{i+j+\op{dim}f}(X',Z) . \endCD
$$
\item For a proper morphism $g:Z' \to Z$, 
$$(\alp \bullet  \be)\, {}^*g = \alp \bullet  \be \, {}^*g \quad (\in \Cal   Z^{i+j}(X,Z')),$$
i.e., the following diagram commutes:
$$\CD
\Cal  Z^i(X,Y)\otimes \Cal  Z^j(Y,Z) @> {\bullet}>> \Cal   Z^{i+j}(X,Z) \\
@V {\op{id}_{\Cal  Z^i(X,Y)} \otimes \, \, {}^*g} VV @VV {{}^*g} V\\
\Cal Z^i(X,Y)\otimes \Cal  Z^j(Y,Z') @>> {\bullet} > \Cal   Z^{i+j}(X,Z') . \endCD
$$ 
\end{enumerate}
\item[($A_{23}$)] {\bf Pushforward and pullback commute}: For $\alp \in \Cal   Z^i(X,Y)$
 \begin{enumerate} 
\item (proper pushforward and proper pullback commute) For proper morphisms $f:X \to X'$ and $g:Y' \to Y$, 
$$(f_*\alp)\, {}^*g = f_*(\alp \, {}^*g) \quad (\in \Cal   Z^i(X',Y')), \, \text{simply denoted by $f_* \alp\, {}^*g$} $$
i.e., the following diagram commutes:
$$\CD
\Cal   Z^i(X,Y) @> {f_*}>> \Cal Z(X',Y)\\
@V {{}^*g} VV @VV {{}^*g} V\\
\Cal Z^i(X,Y')  @>> {f_*} > \Cal Z^i(X',Y') . \endCD
$$ 
\item (smooth pushforward and smooth pullback commute)For smooth morphisms $f:X' \to X$ and $g:Y \to Y'$, 
$$f^*(\alp \, {}_*g)= (f^*\alp)\, {}_*g \quad (\in \Cal   Z^{i+\op{dim}f + \op{dim}g}(X',Y')), \, \text{simply denoted by $f^* \alp\, {}_*g$} $$
i.e., the following diagram commutes:
$$\CD
\Cal   Z^i(X,Y) @> {{}_*g }>> \Cal   Z^{i+\op{dim}g}(X,Y') \\
@V {f^*} VV @VV {f^*} V\\
\Cal Z^{i+\op{dim}f} (X',Y) @>> {{}_*g } > \Cal  Z^{i+\op{dim}f+\op{dim}g}(X',Y') . \endCD
$$ 
\item (proper pushforward and smooth pullback ``commute" in the following sense) For the following fiber square
 $$\CD
\widetilde X @> {\widetilde f}>> X''\\
@V {\widetilde g} VV @VV {g} V\\
X' @>> {f} > X \endCD
$$ 
with $f$ proper and $g$ smooth, we have
$$g^*f_* = \widetilde f_* \widetilde g^*,$$
i.e., the following diagram commutes:
$$\CD
\Cal   Z^i(X',Y)@> {f_*}>> \Cal   Z^i(X,Y)\\
@V {\widetilde g^*} VV @VV {g^*} V\\
\Cal Z^{i+\op{dim}g}(\widetilde X,Y)  @>> {\widetilde f_*} > \Cal  Z^{i+\op{dim}g}(X'',Y) . \endCD
$$
(Note that $\op{dim}\widetilde g = \op{dim}g$.)

\item (smooth pushforward and proper pullback ``commute" in the following sense) For the following fiber square
 $$\CD
\widetilde Y @> {\widetilde f}>> Y''\\
@V {\widetilde g} VV @VV {g} V\\
Y' @>> {f} > Y \endCD
$$ 
with $f$ proper and $g$ smooth, we have
$$(\alp \, {}_*g)\, {}^*f = (\alp \, {}^*\widetilde f)\, {}_*\widetilde g,$$
i.e., the following diagram commutes:
$$\CD
\Cal  Z^i(X,Y'')@> {{}_*g}>> \Cal   Z^{i+\op{dim}}(X,Y) \\
@V {{}^*\widetilde f} VV @VV {{}^*f} V\\
\Cal Z^i(X,\widetilde Y)  @>> {{}_*\widetilde g} > \Cal  Z^{i+\op{dim}g}(X,Y') . \endCD
$$
(Note that $\op{dim}\widetilde g = \op{dim}g$.)
\end{enumerate}
\item[($A_{123}$)] {\bf ``Projection formula"}: 
\begin{enumerate}
\item For a smooth morphism $g:Y \to Y'$ and $\alp \in \Cal   Z^i(X,Y)$ and $\be \in \Cal   Z^j(Y', Z)$,
$$(\alp \, {}_*g) \bullet \be = \alp \bullet g^*\be \quad (\in \Cal   Z^{i+j+ \op{dim}g}(X,Z)) $$
i.e., the following diagram commutes:
$$\CD
\Cal  Z^i(X,Y)^+ \otimes \Cal Z^j(Y',Z) @> {\op{id}_{\Cal  Z^i(X,Y)} \otimes  g^*}>> \Cal  Z^i(X,Y)\otimes \Cal   Z^{j+\op{dim}g}(Y,Z) \\
@V {{}_*g \otimes \op{id}_{\Cal Z^j(Y,Z)}} VV @VV {\bullet} V\\
\Cal   Z^{i+\op{dim}g}(X,Y') \otimes \Cal Z^j(Y',Z) @>> {\bullet} > \Cal  Z^{i+j+\op{dim}g}(X,Z)  . \endCD
$$ 
\item For a proper map $g:Y' \to Y$, $\alp \in \Cal   Z^i(X,Y)$ and $\be \in \Cal   Z^j(Y', Z)$,
$$(\alp \, {}^*g) \bullet \be = \alp \bullet g_*\be \quad (\in \Cal   Z^{i+j}(X,Z)) ,$$
the following diagram commutes:
$$\CD
\Cal  Z^i(X,Y) \otimes \Cal Z^j(Y',Z) @> {\op{id}_{\Cal  Z^i(X,Y)} \otimes  g_*}>> \Cal  Z^i(X,Y) \otimes \Cal  Z^j(Y,Z)  \\
@V {{}^*g \otimes \op{id}_{\Cal   Z^j(Y',Z)}} VV @VV {\bullet} V\\
\Cal  Z^i(X,Y') \otimes \Cal  Z^j(Y',Z) @>> {\bullet} > \Cal   Z^{i+j}(X,Z)  . \endCD
$$ 
\end{enumerate}
\end{enumerate}
\end{pro}
\begin{rem} $\jeden_X:=[X \xleftarrow {\op{id}_X} X \xrightarrow {\op{id}_X} X] \in \Cal   Z^0(X,X)$ satisfies that $\jeden_X \bullet \alp = \alp$ for any element $\alp \in \Cal   Z^*(X, Y)$ and $\be \bullet \jeden_X = \be$ for any element $\be \in \Cal   Z^*(Y, X)$.
\end{rem} 
The following fact is emphasized for a later use.
\begin{lem}(Pushforward-Product Property for Units (abbr. PPPU))\label{pppu} For the following fiber square
$$\xymatrix{
&& V\times_Y W\ar [dl]_{\widetilde{p}} \ar[dr]^{\widetilde{s}} &&\\
& V \ar [dr]_{s} && W \ar [dl]^{p} \\
 & &  Y && }
$$
with $s:V \to Y$ smooth and $p:W \to Y$ proper, we have
$$(\jeden_V \,{}_*s) \bullet p_*\jeden_W = (\widetilde p_* \jeden_{V\times_Y W})\, {}_*\widetilde s =  \widetilde p_* ( \jeden_{V\times_Y W} \, {}_*\widetilde s ) \in \Cal  Z(V,W).$$
Here we have the homomorphisms
$${}_*s:\Cal  Z^*(V,V) \to \Cal  Z^*(V,Y), \quad p_*:\Cal  Z^*(W,W) \to \Cal  Z^*(Y,W)$$
and the following commutative diagram
$$\xymatrix{
&& \Cal  Z^*(V\times_Y W, V\times_Y W) \ar [dl]_{\widetilde{p}_*} \ar[dr]^{{}_*\widetilde{s}} &&\\
& \Cal  Z^*(V,V\times_Y W)  \ar [dr]_{{}_*\widetilde{s}} && \Cal  Z^*(V\times_Y W, W) \ar [dl]^{\widetilde{p}_*} \\
 & &  \Cal  Z^*(V,W). && }
$$
\end{lem}

In the same way as in Levine--Morel's algebraic cobordism, we define Chern operators as follows:
\begin{defn} For a line bundle $L$ over $X$ or a line bundle $M$ over $Y$, we first define the following Chern classes:
$$c_1(L):=[X \xleftarrow {\op{id}_X} X \xrightarrow {\op{id}_X} X; L] \in \Cal Z^1(X,X),$$
$$c_1(M):=[Y \xleftarrow {\op{id}_Y} Y \xrightarrow {\op{id}_Y} Y; M] \in \Cal Z^1(Y,Y).$$
Then the ``Chern class operators" 
$$c_1(L) \bullet: \Cal Z^i(X,Y) \to \Cal Z^{i+1}(X,Y), \quad \bullet \, c_1(M): \Cal Z^i(X,Y) \to \Cal Z^{i+1}(X,Y)$$
are respectively defined by
$$c_1(L) \bullet ([X \xleftarrow {p} V \xrightarrow {s} Y; L_1, L_2, \cdots,L_r]) =[X \xleftarrow {p} V \xrightarrow {s} Y; L_1, L_2, \cdots,L_r, p^*L],$$
$$([X \xleftarrow {p} V \xrightarrow {s} Y; L_1, L_2, \cdots,L_r]) \bullet c_1(M) =[X \xleftarrow {p} V \xrightarrow {s} Y; L_1, L_2, \cdots,L_r, s^*M].$$
\end{defn} 
\begin{lem}\label{ch-op} The above Chern class operators satisfy the following properties.
\begin{enumerate}
\item (identity): If $L$ and $L'$ are line bundles over $X$ and isomorphic and if $M$ and $M'$ are line bundles over $Y$ and isomorphic, then we have
$$c_1(L) \bullet  = c_1(L') \bullet: \Cal Z^i(X,Y) \to \Cal Z^{i+1}(X,Y),$$
$$ \bullet \, c_1(M) =\bullet \, c_1(M') : \Cal Z^i(X,Y) \to \Cal Z^{i+1}(X,Y).$$
\item (commutativity): If $L$ and $L'$ are line bundles over $X$ and if $M$ and $M'$ are line bundles over $Y$, then we have
$$c_1(L) \bullet c_1(L') \bullet  = c_1(L') \bullet c_1(L) \bullet  :\Cal Z^i(X,Y) \to \Cal Z^{i+2}(X,Y),$$
$$\bullet c_1(M)\bullet c_1(M')= \bullet c_1(M')\bullet c_1(M) : \Cal Z^i(X,Y) \to \Cal Z^{i+2}(X,Y).$$
\item (compatibility with product) Let $L$ be a line bundle over $X$ and $N$ be a line bundle over $Z$. For $\alp \in \Cal Z^i(X,Y)$ and $\be \in \Cal Z^j(Y,Z)$, we have
 $$ c_1(L) \bullet (\alp \bullet \be) = \Bigl (c_1(L) \bullet \alp \Bigr ) \bullet \be,$$
 $$(\alp \bullet \be) \bullet c_1(N) = \alp \bullet \Bigl (\be \bullet c_1(N) \Bigr ).$$
\item (compatibility with pushforward = ``projection formula") (which is ``similar" to \cite[Theorem 3.2, (c) (Projection formula)]{Fulton-book}): For a proper map $f:X \to X'$ and a line bundle $L$ over $X'$ and for a smooth map $g:Y \to Y'$ and a line bundle $M$ over $Y'$ we have that for $\alp \in \Cal Z^i(X,Y)$
$$ f_*\bigl (c_1(f^*L) \bullet \alp) \bigr) = c_1(L) \bullet f_*\alp, \,  \bigl  (\alp \bullet c_1(g^*M) \bigr) \, {}_*g = \alp \, {}_*g \bullet  c_1(M).$$
\item (compatibility with pullback =``pullback formula")(which is ``similar" to \cite[Theorem 3.2, (d) (Pull-back)]{Fulton-book}): For a smooth map $f:X' \to X$ and a line bundle $L$ over $X$ and for a proper map $g:Y' \to Y$ and a line bundle $M$ over $Y$ we have that for $\alp \in \Cal Z^i(X,Y)$
$$ f^*\bigl (c_1(L) \bullet \alp \bigr) = c_1(f^*L) \bullet f^*\alp, \, \bigl (\alp \bullet c_1(M) \bigr) \, {}^*g= \alp \, {}^*g \bullet c_1(g^*M).$$
\end{enumerate}
\end{lem} 
\begin{proof} We show only (4) and (5).
First we observe that as to (4) 
\begin{equation}\label{f4} 
f_*c_1(f^*L) = c_1(L)\,{}^*f, \quad c_1(g^*M)\, {}_*g =g^*c_1(M).
\end{equation}
and as to (5)
\begin{equation}\label{f5} 
c_1(f^*L) \, {}_*f = f^*c_1(L), \quad c_1(g^*M)\, {}_*g =c_1(M)\, {}^*g.
\end{equation}
Indeed, the first one of (\ref{f4}) can be seen as follows:
\begin{align*}
f_*c_1(f^*L) & = f_*([X \xleftarrow {\op{id}_X} X \xrightarrow {\op{id}_X} X; f^*L])\\
& = [X' \xleftarrow f X \xrightarrow {\op{id}_X} X; f^*L]\\
& = [X' \xleftarrow {\op{id}_{X'}} X' \xrightarrow {\op{id}_{X'}} X'; L] \,{}^*f \quad \text {(by the definition of ${}^*f$)}\\
& = c_1(L)\,{}^*f
\end{align*}
 Similarly we can see the other three equalities, which are left to the reader.
As to (4), we show the first one: 
\begin{align*}
f_*\bigl (c_1(f^*L) \bullet \alp) \bigr) & = f_* c_1(f^*L) \bullet \alp \quad \text {(by $A_{12}$ (a) )} \\
& = c_1(L)\,{}^*f \bullet \alp \quad \text{(by (\ref{f4})}\\
& = c_1(L) \bullet f_*\alp \quad \text{(by $A_{123}$ (Projection formula) (a) )}
\end{align*}

As to (5), we show the second one:
\begin{align*}
\bigl (\alp \bullet c_1(M) \bigr) \, {}^*g & = \alp \bullet c_1(M)\, {}^*g \quad \text {(by $A_{13}$ (b) )} \\
& = \alp \bullet c_1(g^*M)\, {}_*g\quad \text{(by (\ref{f5})}\\
& = \alp \, {}^*g \bullet c_1(g^*M) \quad \text{(by $A_{123}$ (Projection formula) (b))}
\end{align*}
\end{proof}
As to the compatibility with pullback, we observe the following fact concerning the unit:
\begin{lem}[Pullback Property for Unit (abbr. PPU)]\label{ppu}  For a smooth map $f:X' \to X$ and a line bundle $L$ over $X$ and for a proper map $g:Y' \to Y$ and a line bundle $M$ over $Y$ we have that 
for $\jeden_X \in \Cal Z^i(X,X)$ and $\jeden_Y \in \Cal Z^i(Y,Y)$
$$c_1(f^*L) \bullet (f^*\jeden_X) = (f^*\jeden_X) \bullet c_1(L) \in \Cal Z^{i+\op{dim}f+1}(X',X),$$
$$(\jeden_Y \, {}^*g) \bullet c_1(g^*M) = c_1(M) \bullet (\jeden_Y \, {}^*g) \in \Cal Z^{i+1}(Y,Y'). \qquad $$
\end{lem}
\begin{proof} We prove only the first equality, i.e., the case of smooth maps, since the second equality can be proved in the same way. Since $\jeden_X = [X \xleftarrow {\op{id}_X} X \xrightarrow {\op{id}_X} X]$, it follows from the definition of the smooth pullback $f^*\jeden_X$ (see Definition \ref{ppp of cbc}) that we have $f^*\jeden_X = [X' \xleftarrow {\op{id}_{X'}} X' \xrightarrow {f} X]$:
$$\xymatrix{
& X' \ar[dl]_{\op{id}_{X'}} \ar[r]^{f} & X \ar [dl]^{\op{id}_X} \ar [dr]^{\op{id}_X}\\
X' \ar[r]_f & X &  & X }
$$
Hence we have 
$$(f^*\jeden_X) \bullet c_1(L) =  [X' \xleftarrow {\op{id}_{X'}} X' \xrightarrow {f} X] \bullet c_1(L) = [X' \xleftarrow {\op{id}_{X'}} X' \xrightarrow {f} X; f^*L].$$
On the other hand, since $f^*L$ is a line bundle over $X'$, clearly we have
\begin{align*}
c_1(f^*L) \bullet (f^*\jeden_X) & =  c_1(f^*L) \bullet [X' \xleftarrow {\op{id}_{X'}} X' \xrightarrow {f} X] \\
& = [X' \xleftarrow {\op{id}_{X'}} X' \xrightarrow {f} X; (\op{id}_{X'})^*f^*L]\\
&= [X' \xleftarrow {\op{id}_{X'}} X' \xrightarrow {f} X; f^*L] \qquad \qquad \qquad \qquad \qquad \qquad 
\end{align*}
Thus we obtain $c_1(f^*L) \bullet (f^*\jeden_X) = (f^*\jeden_X) \bullet c_1(L)$.
\end{proof}
If we let $f:X' \to X$ be the identity $\op{id}_X:X \to X$, we get the following
\begin{cor}[Commutativity of the unit and Chern class] \label{u-c}
Let $L$ be a line bundle over $X$. Then we have
$$c_1(L) \bullet \jeden_X = \jeden_X \bullet c_1(L).$$ 
\end{cor}
\begin{rem}\label{rem1} 
The following observations plays key roles later. 
Let $L_1, L_2, \cdots L_r$ be line bundles over $V$.
Then it follows from Lemma \ref{ch-op} (3) that 
$$c_1(L_1) \bullet \Biggl (c_1(L_2) \bullet \biggl ( \cdots \bullet \Bigl ( c_1(L_r) \bullet \alp \Bigr ) \cdots \biggr ) \Biggr) 
= \Bigl (c_1(L_1) \bullet c_1(L_2) \bullet \cdots c_1(L_r) \Bigr ) \bullet \alp,$$
which is simply denoted by $c_1(L_1) \bullet c_1(L_2) \bullet \cdots c_1(L_r) \bullet \alp.$
Then it follows from the definitions that  we have
$$[V \xleftarrow {\op{id}_V} V \xrightarrow {\op{id}_V} V; L_1, L_2, \cdots,L_r]=c_1(L_1) \bullet c_1(L_2) \bullet \cdots \bullet c_1(L_r) \bullet \jeden_V,$$
$$[V \xleftarrow {\op{id}_V} V \xrightarrow {\op{id}_V} V; L_1, L_2, \cdots,L_r]=\jeden_V \bullet c_1(L_1) \bullet c_1(L_2) \bullet \cdots \bullet c_1(L_r).$$
In fact, it follows from Corollary \ref{u-c} that $\jeden_V$ can be placed at any place; 
\begin{align*}
[V \xleftarrow {\op{id}_V} V \xrightarrow {\op{id}_V} V; L_1, & L_2, \cdots,L_r]\\
& =c_1(L_1) \bullet \cdots \bullet c_1(L_j) \bullet \jeden_V \bullet c_1(L_{j+1}) \bullet \cdots \bullet c_1(L_r).
\end{align*}
Hence we have
$$[X \xleftarrow {p} V \xrightarrow {s} Y; L_1, L_2, \cdots,L_r]=p_* \Bigl (c_1(L_1) \bullet c_1(L_2) \bullet \cdots \bullet c_1(L_r) \bullet \jeden_V \Bigr) \, {}_*s,$$
$$[X \xleftarrow {p} V \xrightarrow {s} Y; L_1, L_2, \cdots,L_r]=p_* \Bigl (\jeden_V \bullet c_1(L_1) \bullet (L_2) \bullet \cdots \bullet c_1(L_r) \Bigr) \, {}_*s,$$
and in general
\begin{eqnarray}\label{p-s-rel}
[X \xleftarrow {p} V \xrightarrow {s} Y; L_1, L_2, \cdots,L_r] \hspace{5.8cm} \\
\hspace{2cm} =p_* \Bigl (c_1(L_1) \bullet \cdots \bullet c_1(L_j) \bullet \jeden_V \bullet c_1(L_{j+1}) \bullet \cdots \bullet c_1(L_r) \Bigr) \, {}_*s. \nonumber
\end{eqnarray}\end{rem}

\begin{defn}[Bi-variant theory] An association $\Cal   B$ assigning to a pair $(X,Y)$ a graded abelian group $\Cal   B^*(X,Y)$ is called a \emph{bi-variant theory}
provided that

\noindent
(1) it is equipped with the following three operations

\begin{enumerate}
\item (Product)  \quad $\bullet: \Cal   B^i(X, Y) \times \Cal   B^j(Y,Z) \to \Cal B^{i+j}(X, Z)$

\item (Pushforward) 
\begin{enumerate}
\item For a \emph{proper} map $f:X \to X'$, $f_*:\Cal   B^i(X,Y) \to \Cal   B^i(X',Y)$.
\item For a \emph{smooth} map $g:Y \to Y'$, ${}_*g: \Cal   B^i(X,Y) \to \Cal   B^{i+\op{dim}g}(X,Y')$.
\end{enumerate}

\item (Pullback) 
\begin{enumerate}
\item For a \emph{smooth} map $f:X' \to X$, $f^*: \Cal   B^i(X,Y) \to  \Cal   B^{i+\op{dim}f}(X',Y)$.
\item For a \emph{proper} map $g:Y' \to Y$, ${}^*g: \Cal   B^i(X,Y) \to  \Cal   B^i(X,Y')$.
\end{enumerate}
\end{enumerate}

\noindent
(2) the three operations satisfy the following nine properties as in Proposition \ref{pro-bi-v}:
\begin{enumerate}
\item[($A_1$)] Product is associative.
\item[($A_2$)] Pushforward is functorial. ((a), (b))
\item[($A_2$)'] Proper pushforward and smooth pushforward commute.
\item[($A_3$)] Pullback is functorial.  ((a), (b))
\item[($A_3$)'] Proper pullback and smooth pullback commute.
\item[($A_{12}$)] Product and pushforward commute.  ((a), (b))
\item[($A_{13}$)] Product and pullback commute.  ((a), (b))
\item[($A_{23}$)] Pushforward and pullback commute.  ((a), (b), (c), (d))
\item[($A_{123}$)] Projection formula.  ((a), (b))
\end{enumerate}

\noindent
(3) $\Cal  B$ has units, i.e., there is an element $1_X \in \Cal  B^0(X,X)$ such that $1_X \bullet \alp = \alp$ for any element $\alp \in \Cal   B(X, Y)$ and $\be \bullet 1_X = \be$ for any element $\be \in \Cal   B(Y, X)$.


\noindent
(4) $\Cal  B$ satisfies PPPU and PPU.

\noindent
(5) $\Cal B$ is equipped with the Chern class operators satisfying the properties in Lemma \ref{ch-op}.
\end{defn}
\begin{defn}
Let $\Cal   B, \Cal   B'$ be two bi-variant theories on a category $\Cal V$. A {\it Grothen- dieck transformation} from $\Cal   B$ to $\Cal   B'$, $\ga : \Cal   B \to \Cal   B'$
is a collection of homomorphisms
$\Cal   B(X, Y) \to \Cal   B'(X,Y)$
for a pair $(X,Y)$ in the category $\Cal V$, which preserves the above three basic operations and the Chern class operator: 
\begin{enumerate}
\item $\ga (\alp \bullet_{\Cal   B} \be) = \ga (\alp) \bullet _{\Cal   B'} \ga (\be)$, 
\item $\ga(f_{*}\alp) = f_*\ga (\alp)$ and $\ga(\alp \, {}_*g) = \ga (\alp) \, {}_*g$,
\item $\ga (g^* \alp) = g^* \ga (\alp)$ and $\ga (\alp \, {}^*f) = \ga (\alp) \, {}^*f$,
\item $\ga(c_1(L) \bullet \alp) = c_1(L) \bullet_{\Cal   B} \ga(\alp)$ and $\ga(\alp \bullet c_1(M) ) = \ga(\alp) \bullet_{\Cal   B} c_1(M)$.
\end{enumerate}
\end{defn}
\begin{thm} The above $\Cal  Z^*(-,-)$ is the universal one among bi-variant theories in the following sense.
Given any bi-variant theory $\Cal  B^*(-,-)$, there exists a unique Grothendieck transformation
$$\gamma_{\Cal  B}: \Cal  Z^*(-,-) \to \Cal  B^*(-,-)$$
such that $\gamma_{\Cal  B}(\jeden_V) = 1_V \in \Cal B(V,V)$ for any variety $V$.
\end{thm}
\begin{proof} Let $\Cal B$ be a bi-variant theory. From now on we just simply write $\gamma$ for $\gamma_{\Cal B}$ and $\bullet$ for $\bullet_{\Cal   B}$ unless some possible confusion with those for the theory $\Cal  Z^*$.
Then, using the observation (\ref{p-s-rel}) made in Remark \ref{rem1}, we define
$$\gamma: \Cal Z^i(X,Y) \to \Cal B^i(X,Y)$$
by, for $[X \xleftarrow {p} V \xrightarrow {s} Y; L_1, \cdots, L_r] \in \Cal Z^i(X,Y)$,
\begin{eqnarray}\label{equ1}
\hspace{1cm} \gamma([X \xleftarrow {p}  V \xrightarrow {s} Y; L_1, \cdots, L_r])  \hspace{6cm} \\
 \hspace{1cm} = \gamma \left (p_*\Bigl (c_1(L_1) \bullet \cdots \bullet c_1(L_j) \bullet \jeden_V \bullet c_1(L_{j+1}) \bullet \cdots \bullet c_1(L_r) \Bigr) \, {}_*s. \right )  \nonumber  \\
:= p_*\Bigl (c_1(L_1) \bullet \cdots \bullet c_1(L_j) \bullet 1_V \bullet c_1(L_{j+1}) \bullet \cdots \bullet c_1(L_r) \Bigr) \, {}_*s.  \hspace{0.8cm} \nonumber
\end{eqnarray}
We show that this transformation satisfies the above four properties:

\noindent
(1) $\ga (\alp \bullet \be) = \ga (\alp) \bullet \ga (\be)$: 

Let 
$$\alp= [X \xleftarrow{p_1} V_1  \xrightarrow{s_1} Y; L_1,\cdots, L_r] \in \Cal Z^i(X,Y), $$
$$ \be = [Y \xleftarrow{p_2} V_2  \xrightarrow{s_2} Z; M_1, \cdots, M_k]  \in \Cal Z^j(Y,Z).$$
Then by the definition we have
\begin{align*}
& [X \xleftarrow{p_1} V_1  \xrightarrow{s_1} Y; L_1, L_2, \cdots, L_r] \bullet [Y \xleftarrow{p_2} V_2  \xrightarrow{s_2} Z; M_1, M_2, \cdots, M_k] \\
& := [(X \xleftarrow{p_1} V_1  \xrightarrow{s_1} Y) \circ (Y \xleftarrow{p_2} V_2  \xrightarrow{s_2} Z) ; \widetilde{p_2}^*L_1, \cdots, \widetilde{p_2}^*L_r, \widetilde{s_1}^*M_1, \cdots \widetilde{s_1}^*M_k]\\
& = [X \xleftarrow{p_1\circ \widetilde{p_2}} V_1\times_Y V_2  \xrightarrow{s_2 \circ \widetilde{s_1}} Z; \widetilde{p_2}^*L_1, \cdots, \widetilde{p_2}^*L_r, \widetilde{s_1}^*M_1, \cdots \widetilde{s_1}^*M_k]
\end{align*}
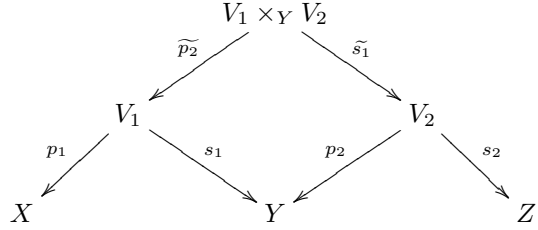
$$\xymatrix{
&& V_1\times_Y V_2 \ar [dl]_{\widetilde{p_2}} \ar[dr]^{\widetilde{s_1}} &&\\
& V_1 \ar [dl]_{p_1} \ar [dr]^{s_1} && V_2 \ar [dl]_{p_2} \ar[dr]^{s_2}\\
X & &  Y && Z }
$$
Hence 
we have
\begin{equation*}\begin{split}
 & \ga (\alp \bullet \be) \\
 & =\ga([X \xleftarrow{p_1\circ \widetilde{p_2}} V_1\times_Y V_2  \xrightarrow{s_2 \circ \widetilde{s_1}} Z; \widetilde{p_2}^*L_1, \cdots, \widetilde{p_2}^*L_r, \widetilde{s_1}^*M_1, \cdots \widetilde{s_1}^*M_k]) \qquad \\
 \end{split}
\end{equation*}
\begin{equation*}\begin{split}
 & = \ga \biggl  ( (p_1\circ \widetilde{p_2})_*\Bigl( c_1(\widetilde{p_2}^*L_1)\bullet \cdots c_1(\widetilde{p_2}^*L_r)  \bullet \jeden_{V_1 \times_Y V_2} \\
& \hspace{6cm} \bullet c_1(\widetilde{s_1}^*M_1) \cdots \bullet c_1(\widetilde{s_1}^*M_k)\Bigr) \, {}_*(s_2 \circ \widetilde{s_1})  \biggr ) \\
& \overset{(\ref{equ1})}{=}  (p_1\circ \widetilde{p_2})_*\Bigl( c_1(\widetilde{p_2}^*L_1)\bullet \cdots c_1(\widetilde{p_2}^*L_r)  \bullet 1_{V_1 \times_Y V_2} \\
& \hspace{6cm} \bullet c_1(\widetilde{s_1}^*M_1) \cdots \bullet c_1(\widetilde{s_1}^*M_k)\Bigr) \, {}_*(s_2 \circ \widetilde{s_1}) \\
& \overset{(A_2)}{=} (p_1)_*(\widetilde{p_2})_* \Bigl( c_1(\widetilde{p_2}^*L_1)\bullet \cdots c_1(\widetilde{p_2}^*L_r) \bullet 1_{V_1 \times_Y V_2} \\
& \hspace{6cm} \bullet c_1(\widetilde{s_1}^*M_1) \cdots \bullet c_1(\widetilde{s_1}^*M_k) \Bigr) \, {}_*(\widetilde{s_1}){}_*(s_2)\\
& \overset{(A_2)'}{=} 
(p_1)_* \biggl ((\widetilde{p_2})_* \Bigl ( c_1(\widetilde{p_2}^*L_1)\bullet \cdots c_1(\widetilde{p_2}^*L_r) \bullet 1_{V_1 \times_Y V_2} \\
& \hspace{5.5cm}  \bullet c_1(\widetilde{s_1}^*M_1) \cdots \bullet c_1(\widetilde{s_1}^*M_k) \Bigr ) \, {}_*(\widetilde{s_1}) \biggr) {}_*(s_2)
\end{split}
\end{equation*}
Then by applying the property (4) of Lemma \ref{ch-op} successively with respect to line bundles $L_1, \cdots , L_r$ and $M_1, \cdots , M_k $ we get
\begin{align*}
 (\widetilde{p_2})_* \Bigl ( & c_1(\widetilde{p_2}^*L_1) \bullet \cdots c_1(\widetilde{p_2}^*L_r) \bullet 1_{V_1 \times_Y V_2} \bullet c_1(\widetilde{s_1}^*M_1) \cdots \bullet c_1(\widetilde{s_1}^*M_k) \Bigr ) \, {}_*(\widetilde{s_1}) \\
& =  c_1(L_1)\bullet \cdots c_1(L_r) \bullet (\widetilde{p_2})_* 1_{V_1 \times_Y V_2} \, {}_*(\widetilde{s_1})\bullet c_1(M_1) \cdots \bullet c_1(M_k).
\end{align*}
Here it follows from Lemma \ref{pppu} (PPPU) that we have
$$ (\widetilde{p_2})_* 1_{V_1 \times_Y V_2} \, {}_*(\widetilde{s_1}) = 1_{V_1} \, {}_*(s_1) \bullet (p_2)_* 1_{V_2}.$$
Hence the above equalities continue as follows:
\begin{equation*}\begin{split}
& =(p_1)_*\Bigl (c_1(L_1)\cdots c_1(L_r) \bullet 1_{V_1} \, {}_*(s_1) \bullet (p_2)_* 1_{V_2} \bullet c_1(M_1) \cdots c_1(M_k)\Bigr )\, {}_* (s_2) \\
& \overset{(A_2)'}{=} \Bigl ((p_1)_*\bigl (c_1(L_1)\cdots c_1(L_r) \bullet 1_{V_1} \, {}_*(s_1) \\
& \hspace{5.8cm} \bullet (p_2)_* 1_{V_2} \bullet c_1(M_1) \cdots c_1(M_k)\bigr ) \Bigr ) \, {}_* (s_2) \\
& \overset{(A_{12})(a)}{=}  \biggl ((p_1)_*\Bigl ( c_1(L_1)\cdots c_1(L_r) \bullet 1_{V_1} \, {}_*(s_1) \Bigr ) \\
& \hspace{5.5cm} \bullet \Bigl ((p_2)_* 1_{V_2} \bullet c_1(M_1) \cdots c_1(M_k) \Bigr ) \biggr ) \, {}_* (s_2) \\
& \overset{(A_{12})(b)}{=}  (p_1)_*\Bigl ( c_1(L_1)\cdots c_1(L_r) \bullet 1_{V_1} \, {}_*(s_1) \Bigr ) \\
& \hspace{5.7cm} \bullet \Bigl ((p_2)_* 1_{V_2} \bullet c_1(M_1) \cdots c_1(M_k) \Bigr ) \, {}_* (s_2) \\
& \overset{(A_{12})}{=}  \Bigl ( (p_1)_*\bigl (c_1(L_1)\cdots c_1(L_r) \bullet 1_{V_1} \bigr) \, {}_*(s_1) \Bigr ) \\
& \hspace{5.5cm} \bullet \Bigl ((p_2)_* \bigl (1_{V_2} \bullet c_1(M_1) \cdots c_1(M_k) \bigr ) \, {}_* (s_2) \Bigr) \\
\end{split}
\end{equation*}
\begin{equation*}\begin{split}
& =\ga([X \xleftarrow{p_1} V_1  \xrightarrow{s_1} Y; L_1, \cdots, L_r]) \bullet \ga([Y \xleftarrow{p_2} V_2  \xrightarrow{s_2} Z; M_1, \cdots, M_k])\\
&= \ga(\alp) \bullet\ga(\be).
\end{split}
\end{equation*}
Hence we have $\ga (\alp \bullet \be)=\ga(\alp) \bullet\ga(\be).$\\

\noindent
(2) $\ga(f_{*}\alp) = f_*\ga (\alp)$ and $\ga(\alp \, {}_*f) = \ga (\alp) \, {}_*f$: 

Let $\alp =[X \xleftarrow {p} V \xrightarrow {s} Y; L_1, L_2, \cdots, L_r] \in \Cal Z^i(X,Y)$.

(i) $\ga(f_{*}\alp) = f_*\ga (\alp)$: Let $f:X \to X'$ be a proper map. Then we have
\begin{equation*}\begin{split}
\ga(f_*\alp) & = \ga([X' \xleftarrow {f \circ p} V \xrightarrow {s} Y; L_1, \cdots, L_r])\\
&= \ga \left (  (f \circ p)_*\Bigl (c_1(L_1) \bullet \cdots c_1(L_r) \bullet \jeden _V \Bigr) \, {}_*s \right )\\
& \overset{(\ref{equ1})}{=}  (f \circ p)_*\Bigl (c_1(L_1) \bullet \cdots c_1(L_r) \bullet 1_V \Bigr) \, {}_*s\\
& \overset{(A_2)(a)}{=}  f_* p_*\Bigl (c_1(L_1) \bullet \cdots c_1(L_r) \bullet 1_V \Bigr) \, {}_*s\\
& \overset{(A_2)'}{=}  f_* \Biggl (p_*\Bigl (c_1(L_1) \bullet  \cdots c_1(L_r) \bullet 1_V \Bigr) \, {}_*s \Biggr)\\
& = f_*\ga([X \xleftarrow {p} V \xrightarrow {s} Y; L_1, L_2, \cdots, L_r] )\\
& = f_*\ga(\alp).
\end{split}
\end{equation*}

(ii) $\ga(\alp \, {}_*f) = \ga (\alp) \, {}_*f$: Let $f:Y \to Y'$ be a smooth map. Then we have
\begin{equation*}\begin{split}
\ga(\alp \, {}_*f) & = \ga([X \xleftarrow {p} V \xrightarrow {f \circ s} Y'; L_1, L_2, \cdots, L_r])\\
& = \ga \left ( p_*\Bigl (c_1(L_1) \bullet \cdots c_1(L_r) \bullet \jeden_V\Bigr)\, {}_*(f \circ s) \right) \\
&\overset{(\ref{equ1})}{=} p_*\Bigl (c_1(L_1) \bullet \cdots c_1(L_r) \bullet 1_V\Bigr)\, {}_*(f \circ s)\\
& \overset{(A_2)(b)}{=}  p_*\Bigl (c_1(L_1) \bullet \cdots c_1(L_r) \bullet 1_V \Bigr)\, {}_*s \, {}_*f\\
& \overset{(A_2)'}{=} \left (p_*\Bigl (c_1(L_1) \bullet  \cdots c_1(L_r) \bullet 1_V \Bigr) \, {}_*s \right) \, {}_*f\\
& = \ga([X \xleftarrow {p} V \xrightarrow {s} Y; L_1, L_2, \cdots, L_r] ) \, {}_*f\\
& = \ga(\alp) \, {}_*f.
\end{split}
\end{equation*}

\noindent
(3) $\ga (f^* \alp) = f^* \ga (\alp)$ and $\ga (\alp \, {}^*f) = \ga (\alp) \, {}^*f$: 

Let $\alp =[X \xleftarrow {p} V \xrightarrow {s} Y; L_1, L_2, \cdots, L_r] \in \Cal Z^i(X,Y)$.

(i) $\ga (f^* \alp) = f^* \ga (\alp)$: Let $f:X' \to X$ be a smooth map and consider the following diagram:
$$\xymatrix{
& X' \times_X V \ar[dl]_{p'} \ar[r]^{\qquad f'} & V \ar [dl]^{p} \ar [dr]^{s}\\
X' \ar[r]_f & X &  & Y }
$$
\begin{equation*}\begin{split}
\ga(f^*\alp) & = \ga([X' \xleftarrow {p'} X' \times_X V \xrightarrow {s \circ f'} Y; (f')^*L_1, (f')^*L_2, \cdots, (f')^*L_r])\\
& = \ga \left ( (p')_*\biggl (\jeden_{X' \times_X V} \bullet c_1\Bigl((f')^*L_1 \Bigr) \cdots \bullet c_1 \Bigl((f')^*L_r \Bigr) \biggr) \, {}_*(s \circ f') \right) \\
& \overset{(\ref{equ1})}{=} (p')_*\biggl (1_{X' \times_X V} \bullet c_1\Bigl((f')^*L_1 \Bigr) \cdots \bullet c_1 \Bigl((f')^*L_r \Bigr) \biggr) \, {}_*(s \circ f') \\
& \overset{(A_2)(b) }{=}  (p')_*\biggl (1_{X' \times_X V} \bullet c_1\Bigl((f')^*L_1 \Bigr) \cdots \bullet c_1 \Bigl((f')^*L_r  \Bigr ) \biggr) \, {}_*(f') \, {}_*s\\
& =  \Biggl ( (p')_* \biggl (1_{X' \times_X V} \bullet c_1\Bigl((f')^*L_1 \Bigr) \cdots \bullet c_1 \Bigl((f')^*L_r  \Bigr ) \biggr) \, {}_*(f') \Biggr) \, {}_*s\\
\end{split}
\end{equation*}
By applying the property (4) of Lemma \ref{ch-op} successively with respect to line bundles $L_1$, $L_2$, $\cdots , L_r$, the above equalities continue as follows:
\begin{equation*}\begin{split}
& = \biggl ( (p')_* \Bigl (1_{X' \times_X V}\, {}_*(f') \bullet c_1(L_1) \cdots \bullet c_1(L_r) \Bigr ) \biggr ) \, {}_*s\\
& \overset{(A_{12})(a)} {=} \Bigl ((p')_*1_{X' \times_X V}\, {}_*(f') \bullet c_1(L_1) \cdots \bullet c_1(L_r) \Bigr ) 
\, {}_*s
\end{split}
\end{equation*}
Here we note that
\begin{align*}
(p')_*1_{X' \times_X V}\, {}_*(f')  & = 1_{X'}\, {}_*f \bullet p_* 1_V \quad \text {(by Lemma \ref{pppu} (PPPU))}\\
& = 1_{X'} \bullet f^*p_* 1_V \quad \text {(by $A_{123}$ (a))}
\end{align*}
Thus the above equalities continue as follows:
\begin{equation*}\begin{split}
& = \Bigl (1_{X'} \bullet f^*p_* 1_V \bullet c_1(L_1) \cdots \bullet c_1(L_r) \Bigr ) 
\, {}_*s
\\
& = \Bigl (f^*p_* 1_V \bullet c_1(L_1) \cdots \bullet c_1(L_r) \Bigr ) 
\, {}_*s \quad \text{(since $1_{X'}$ is the unit)} \\
& \overset{(A_{13})(a)}{=} \Bigl (f^* \bigl (p_* 1_V \bullet c_1(L_1) \cdots \bullet c_1(L_r) \bigr ) \Bigr ) 
\, {}_*s\\
& \overset{(A_{23})(b)}{=} f^* \bigl (p_* 1_V \bullet c_1(L_1) \cdots \bullet c_1(L_r) \bigr )\, {}_*s \\
& \overset{(A_{12})(a)}{=}f^* \Bigl (p_* (1_V \bullet c_1(L_1) \cdots \bullet c_1(L_r) ) \Bigr) \, {}_*s \\
& \overset{(A_{23})(b)}{=} f^* \Bigl (p_* (1_V \bullet c_1(L_1) \cdots \bullet c_1(L_r) ) \, {}_*s  \bigr) \\
& = f^* \ga([X \xleftarrow {p} V \xrightarrow {s} Y; L_1, L_2, \cdots, L_r] ) \\
& = f^* \ga(\alp)
\end{split}
\end{equation*}
Here we note that in fact we can show the above ``indirectly", using Lemma \ref{ppu} (PPU) and Lemma \ref{ch-op} (5) as well:
\begin{equation*}\begin{split}
& = \Bigl (f^*p_* 1_V \bullet c_1(L_1) \bullet c_1(L_2) \bullet \cdots \bullet c_1(L_r) \Bigr ) 
\, {}_*s  \\
& \overset{(A_{23})(c)}{=} \Bigl ((p')_*(f')^* 1_V \bullet c_1(L_1) \bullet c_1(L_2) \bullet \cdots \bullet c_1(L_r) \Bigr ) 
\, {}_*s  \\
\end{split}
\end{equation*}
\begin{equation*}\begin{split}
& \overset{(A_2)'}{=}  (p')_* \Bigl ( \underbrace{(f')^* 1_V \bullet c_1(L_1)} \bullet c_1(L_2) \bullet \cdots \bullet c_1(L_r) \Bigr ) 
\, {}_*s  \\
& =   (p')_* \Bigl (c_1((f')^*L_1) \bullet \underbrace{(f')^* 1_V \bullet c_1(L_2)} \bullet  \cdots \bullet c_1(L_r) \Bigr ) 
\, {}_*s  \\
& \hspace{8cm} \text{(by applying (PPU) for $L_1$)} \\
& =   (p')_* \Bigl (c_1((f')^*L_1) \bullet c_1((f')^*L_2) \bullet \underbrace{(f')^* 1_V \bullet c_1(L_3)} \bullet  \cdots \bullet c_1(L_r)\Bigr ) 
\, {}_*s  \\
& \hspace{8cm} \text{(by applying (PPU) for $L_2$)} \\
& =   (p')_* \Bigl (c_1((f')^*L_1) \bullet c_1((f')^*L_2) \bullet  \cdots \bullet \underbrace{c_1((f')^*L_r) \bullet (f')^* 1_V} \Bigr ) 
\, {}_*s  \\
&  \hspace{7.3cm} \text{(by applying (PPU) for the rest)} \\
& =   (p')_* \Bigl (c_1((f')^*L_1) \bullet c_1((f')^*L_2) \bullet \cdots \bullet (f')^*(c_1(L_r) \bullet 1_V)  \Bigr ) 
\, {}_*s  \\
& \hspace{6cm} \quad \text{(by applying Lemma \ref{ch-op} (5) for $L_r$)} \\
& =   (p')_* \Bigl ((f')^* \bigl (c_1(L_1) \bullet \cdots \bullet c_1(L_r) \bullet 1_V) \bigr ) \Bigr ) 
\, {}_*s  \\
& \hspace{5cm}  \quad \text{(by applying Lemma \ref{ch-op} (5) successively)} \\
& \overset{(A_2)'}{=} \Bigl ((p')_*(f')^* \bigl (c_1(L_1) \bullet \cdots \bullet c_1(L_r) \bullet 1_V) \bigr ) \Bigr ) \, {}_*s  \\
& \overset{(A_{23})(c)}{=} \Bigl (f^*p_* \bigl (c_1(L_1) \bullet \cdots \bullet c_1(L_r) \bullet 1_V \bigr ) \Bigr ) \, {}_*s  \\
& \overset{(A_{23})(b)}{=} f^* \Bigl (p_* \bigl (c_1(L_1) \bullet \cdots \bullet c_1(L_r) \bullet 1_V \bigr ) {}_*s \Bigr )  \\
& =  f^* \ga([X \xleftarrow {p} V \xrightarrow {s} Y; L_1, L_2, \cdots, L_r] ) \\
& = f^* \ga(\alp)
\end{split}
\end{equation*}

(ii) $\ga (\alp \, {}^*f) = \ga (\alp) \, {}^*f$: Let $f:Y' \to Y$ be a proper map. The proof is similar to the above. For the sake of the reader's convenience, we write down the proof.
Let $f:Y' \to Y $ be a proper map and consider the following diagram:
$$\xymatrix{
& V \ar [dl]^{p} \ar [dr]_{s} & V \times _Y Y' \ar[l]_{f' \quad } \ar[dr]^{s'}\\
X &  & Y & Y' \ar[l]^f}
$$
\begin{equation*}\begin{split}
\ga(\alp \, {}^*f) & = \ga([X \xleftarrow {p\circ f'} X' \times_X V \xrightarrow {s'} Y'; (f')^*L_1, (f')^*L_2, \cdots, (f')^*L_r])\\
& = (p \circ f')_*\Biggl (c_1\Bigl((f')^*L_1 \Bigr)\bullet \cdots c_1\Bigl ((f')^*L_r \Bigr) \bullet 1_{V \times_Y Y'}) \Biggr) \, {}_*(s')\\
& \overset{(A_2)}{=} p_*(f')_*\Biggl (c_1\Bigl((f')^*L_1 \Bigr)\bullet \cdots c_1\Bigl ((f')^*L_r \Bigr) \bullet 1_{V \times_Y Y'}) \Biggr) \, {}_*(s')\\
\end{split}
\end{equation*}
By applying the property (4) of Lemma \ref{ch-op} successively with respect to line bundles $L_1$, $L_2$, $\cdots , L_r$, the above equalities continue as follows:
\begin{equation*}\begin{split}
& = p_* \Bigl  (c_1(L_1) \bullet \cdots c_1(L_r) \bullet (f')_*1_{V \times_Y Y'} \Bigr) \, {}_*(s')\\
& \overset{(A_{12})}{=} p_* \Bigl  (c_1(L_1) \bullet \cdots c_1(L_r) \bullet (f')_*1_{V \times_Y Y'}  \, {}_*(s') \Bigr)
\end{split}
\end{equation*}
Here it follows from  Lemma \ref{pppu} (PPPU)) that we have
$$(f')_*1_{V \times_Y Y'}= 1_V \, {}_*s\bullet f_*1_{Y'}.$$
Thus the above equalities continue as follows:
\begin{equation*}\begin{split}
& = p_* \Bigl  (c_1(L_1) \bullet \cdots c_1(L_r) \bullet (1_V \, {}_*s \bullet f_*1_{Y'} )\Bigr)\\
& \overset{(A_{123})(b)}{=}  p_* \Bigl  (c_1(L_1) \bullet \cdots c_1(L_r) \bullet (1_V \, {}_*s \, {}^*f \bullet 1_{Y'} )\Bigr)\\
& = p_* \Bigl  (c_1(L_1) \bullet \cdots c_1(L_r) \bullet 1_V \, {}_*s \, {}^*f \Bigr) \quad \text{(since $1_{Y'}$ is the unit)} \\
& \overset{(A_{13})}{=}   p_* \Bigl  (\bigl (c_1(L_1) \bullet \cdots c_1(L_r) \bullet 1_V \, {}_*s \bigr) \, {}^*f \Bigr) \\
& \overset{(A_{2})'}{=}    \Bigl  (p_*\bigl (c_1(L_1) \bullet \cdots c_1(L_r) \bullet 1_V \, {}_*s \bigr) \Bigr) \, {}^*f \\
&= \Bigl  (p_* \bigl (c_1(L_1) \bullet \cdots c_1(L_r) \bullet 1_V \bigr) \, {}_*s \Bigr) \, {}^*f \\
& = \ga(\alp)\, {}^*f .
\end{split}
\end{equation*}
As done for another proof of (i) using using Lemma \ref{ppu} (PPU) and Lemma \ref{ch-op} (5), we can show the above in a similar way. For the sake of the reader we write down the proof.
\begin{equation*}\begin{split}
& = p_* \Bigl  (c_1(L_1) \bullet \cdots \bullet c_1(L_{r-1}) \bullet c_1(L_r) \bullet 1_V \, {}_*s \, {}^*f \Bigr) \\
& \overset{(A_{23})(d)}{=}   p_* \Bigl  (c_1(L_1) \bullet \cdots \bullet c_1(L_{r-1}) \bullet c_1(L_r) \bullet 1_V \, {}^*(s') \, {}_*(f') \Bigr)
 \\
 & \overset{(A_2)'}{=}   p_* \Bigl  (c_1(L_1) \bullet \cdots \bullet c_1(L_{r-1}) \bullet  \underbrace{c_1(L_r) \bullet 1_V \, {}^*(s')} \Bigr) \, {}_*(f') 
 \\
& =    p_* \Bigl  (c_1(L_1) \bullet \cdots \bullet \underbrace{c_1(L_{r-1}) \bullet 1_V \, {}^*(s')} \bullet c_1((s')^*L_r) \Bigr) \, {}_*(f') \\
& \hspace{8cm} \quad \text{(by applying (PPU) for $L_r$)} \\
& =    p_* \Bigl  (c_1(L_1) \bullet \cdots \bullet 1_V \, {}^*(s') \bullet c_1((s')^*L_{r-1})  \bullet c_1((s')^*L_r) \Bigr) \, {}_*(f') \\
& \hspace{7.6cm} \quad \text{(by applying (PPU) for $L_{r-1}$)} \\
& =    p_* \Bigl  (\underbrace{1_V \, {}^*(s') \bullet c_1((s')^*L_1)} \bullet \cdots \bullet c_1((s')^*L_r) \Bigr) \, {}_*(f') \\
& \hspace{7.3cm} \quad \text{(by applying (PPU) for the rest)} \\    
& =    p_* \Bigl  (\bigl ( 1_V \bullet c_1(L_1) \bigr ) \, {}^*(s') \bullet \cdots \bullet c_1((s')^*L_r) \Bigr) \, {}_*(f') \\
& \hspace{6.5cm} \quad \text{(by applying Lemma \ref{ch-op} (5)  for $L_1$)} \\ 
& =    p_* \Bigl  (\bigl ( 1_V \bullet c_1(L_1) \bullet \cdots \bullet c_1(L_r) \bigr)  \, {}^*(s') \Bigr) \, {}_*(f') \\
& \hspace{5cm} \quad \text{(by applying Lemma \ref{ch-op} (5) successively)} \\
& \overset{(A_{2})'}{=}     p_* \Bigl  (\bigl ( 1_V \bullet c_1(L_1) \bullet \cdots \bullet c_1(L_r) \bigr)  \, {}^*(s') \, {}_*(f') \Bigr)  \\
& \overset{(A_{23})(d)}{=}     p_* \Bigl  (\bigl ( 1_V \bullet c_1(L_1) \bullet \cdots \bullet c_1(L_r) \bigr)  \, {}_*s \, {}^*f  \Bigr)  \\
& \overset{(A_{23})(a)}{=} \Bigl  (p_* \bigl (1_V \bullet c_1(L_1) \bullet \cdots c_1(L_r) \bigr) \, {}_*s \Bigr) \, {}^*f \\
& = \ga(\alp)\, {}^*f .
\end{split}
\end{equation*}

\noindent
(4) $\ga(c_1(L) \bullet \alp) = c_1(L) \bullet \ga(\alp)$ and  $\ga(\alp \bullet c_1(M)) = \ga(\alp) \bullet c_1(M)$:
Let $\alp =[X \xleftarrow {p} V \xrightarrow {s} Y; L_1, L_2, \cdots, L_r] \in \Cal Z^i(X,Y)$.

(i) $\ga(c_1(L) \bullet \alp) = c_1(L) \bullet \ga(\alp)$: Let $L$ be a line bundle over $X$. Then we have
\begin{equation*}\begin{split}
\ga( c_1(L) \bullet \alp)  & = \ga([X \xleftarrow {p} V \xrightarrow {s} Y; L_1, \cdots, L_r, p^*L])\\
& = p_*\Bigl (c_1(L_1) \bullet \cdots c_1(L_r) \bullet c_1(p^*L) \bullet 1_V \Bigr) \, {}_*s\\
& \overset{\text{(Lemma \ref{ch-op} (2))}}{=} p_*\Bigl (c_1(p^*L) \bullet c_1(L_1) \bullet \cdots c_1(L_r) \bullet 1_V \Bigr) \, {}_*s\\
& = \biggl (p_*\Bigl (c_1(p^*L) \bullet c_1(L_1) \bullet \cdots c_1(L_r) \bullet 1_V \Bigr) \biggr )\, {}_*s\\
& \overset{\text{(Lemma \ref{ch-op} (4))}}{=} \biggl (c_1(L) \bullet p_* \Bigl (c_1(L_1) \bullet \cdots c_1(L_r) \bullet 1_V \Bigr) \biggr )\, {}_*s\\
& = c_1(L) \bullet \biggl (p_* \Bigl (c_1(L_1) \bullet \cdots c_1(L_r) \bullet 1_V \Bigr) \biggr )\, {}_*s\\
& = c_1(L)\bullet \ga(\alp)
\end{split}
\end{equation*}

(ii) $\ga(\alp \bullet c_1(M)) = \ga(\alp) \bullet c_1(M)$: Let $M$ be a line bundle over $Y$. Then we have
\begin{equation*}\begin{split}
\ga(\alp \bullet c_1(M))  & = \ga([X \xleftarrow {p} V \xrightarrow {s} Y; L_1, L_2, \cdots, L_r, s^*M])\\
& = p_*\Bigl (c_1(L_1)\cdots \bullet c_1(L_r)\bullet 1_V \bullet c_1(s^*M) \Bigr)\,{}_*s \, \, \,  \, \text{(see Rmeark \ref{rem1})} \\
& \overset{(A_2)'}{=} p_* \biggl ( \Bigl (c_1(L_1)\cdots \bullet c_1(L_r)\bullet 1_V \bullet c_1(s^*M) \Bigr)\,{}_*s \biggr )\\
\end{split}
\end{equation*}
\begin{equation*}\begin{split}
& \overset{\text{(Lemma \ref{ch-op} (4))}}{=} p_* \biggl ( \Bigl (c_1(L_1)\cdots \bullet c_1(L_r) \bullet 1_V \Bigr) \, {}_*s  \bullet c_1(M) \biggr )\\
& \overset{(A_{12})}{=}  \biggl (p_* \Bigl (c_1(L_1)\cdots \bullet c_1(L_r) \bullet 1_V \Bigr) \, {}_*s \biggr ) \bullet c_1(M) \\
& = \ga(\alp) \bullet c_1(M)
\end{split}
\end{equation*}
(5) The uniqueness of the Grothendieck transformation $\ga:\Cal Z \to \Cal B$ follows from the compatibility of pushforward and the Chern class operator and the requirement that the unit is mapped to the unit. Indeed, let $\ga':\Cal Z \to \Cal B$ be a such a Grothendieck transformation.
Then we have
\begin{align*}
\ga'(\alp) & = \ga'([X \xleftarrow {p} V \xrightarrow {s} Y; L_1, L_2, \cdots, L_r]) \\
& = \ga'\Biggl (p_*\Bigl (c_1(L_1) \bullet \cdots c_1(L_r) \bullet \jeden_V \Bigr)\, {}_*s  \Biggr)\\
& = p_*\Bigl (c_1(L_1) \bullet \cdots c_1(L_r) \bullet \ga'(\jeden_V) \Bigr)\, {}_*s  \\
& = p_*\Bigl (c_1(L_1) \bullet \cdots c_1(L_r) \bullet 1_V) \Bigr)\, {}_*s  \\
& = \ga_{\Cal B}(\alp).
\end{align*}
\end{proof}

\begin{rem} The ``forget" map defined in (\ref{forget}) gives rise to the following canonical homomorphism 
$$\frak f: \mathbb {OM}^{pro}_{sm}(X \xrightarrow f Y) \to \Cal  Z^*(X,Y)$$
 defined by
$$\frak f ([V \xrightarrow p Y; L_1, \cdots, L_r]):= [X \xleftarrow p V \xrightarrow {f \circ p} Y; L_1, \cdots, L_r].$$
Then it follows from the definitions that the following diagrams are commutative: 
\begin{enumerate}
\item As to the product $\bullet$:
$$\CD
\mathbb {OM}^{prop}_{sm}(X \xrightarrow f Y)  \otimes \mathbb {OM}^{prop}_{sm}(Y \xrightarrow g Z)  @> {\bullet}>> 
\mathbb {OM}^{prop}_{sm}(X \xrightarrow {g \circ f} Z)  \\
@V {\frak f  \otimes \frak f} VV @VV {\frak f} V\\
\Cal Z^i(X,Y) \otimes \Cal  Z^j(Y,Z) @>> {\bullet} > \Cal   Z^{i+j}(X,Z) . \endCD
$$ 
\item As to the pushforward:
$$\CD
\mathbb {OM}^{prop}_{sm}(X \xrightarrow {g \circ f} Z)  @> {f_*}>> 
\mathbb {OM}^{prop}_{sm}(Y \xrightarrow {g} Z)  \\
@V {\frak f} VV @VV {\frak f} V\\
\Cal Z^*(X,Z)@>> {f_*} > \Cal   Z^*(Y,Z) . \endCD 
$$ 

\item As to the Chern class operator:
$$\CD
\mathbb {OM}^{prop}_{sm}(X \xrightarrow f Y)  @> {\Phi(L)}>> 
\mathbb {OM}^{prop}_{sm}(X \xrightarrow f Y)  \\
@V {\frak f} VV @VV {\frak f} V\\
\Cal Z^*(X,Y)@>> {c_1(L)\bullet} > \Cal   Z^*(X,Y) . \endCD
$$ 
\end{enumerate}
As to the pullback we cannot expect a canonical commutative diagram. Indeed we consider the following fiber square

$$\CD
V' @> g'' >> V \\
@V {h'} VV @VV {h}V\\
X' @> g' >> X \\
@V f' VV @VV f V\\
Y' @>> g > Y. \endCD
$$

Then we have the following diagram, which does not necessarily commute:
\begin {equation}\label{dia}
\xymatrix
{
\mathbb {OM}^{prop}_{sm}(X \xrightarrow f Y) \ar[r]^{g^*} \ar[d]_{\frak f} &  \mathbb {OM}^{prop}_{sm}(X' \xrightarrow {f'} Y') \ar[d]_{\frak f}\\
\Cal Z^*(X,Y)  \ar[r]_{(g')^* \circ (\, {}^*g)} & \Cal Z^*(X',Y') 
}
\end{equation}
\begin{align*}
(\frak f \circ g^*)([V \xrightarrow h X;& L_1, \cdots, L_r]) \\
&= \frak f ([V' \xrightarrow h' X'; (g'')^*L_1, \cdots, (g'')^*L_r])\\
& =[X' \xleftarrow h' V' \xrightarrow {f' \circ h'} Y'; (g'')^*L_1, \cdots, (g'')^*L_r])
\end{align*}
\begin{align*}
 & \Bigl  (\bigl ((g')^* \circ (\, {}^*g) \bigr )\circ \frak f \Bigr )  ([V \xrightarrow h X; L_1, \cdots, L_r]) \\
&= \bigl ((g')^* \circ (\, {}^*g) \bigr ) ([X \xleftarrow h V \xrightarrow {f \circ h} Y; L_1, \cdots, L_r])\\
&= (g')^* ([X \xleftarrow h V \xrightarrow {f \circ h} Y; L_1, \cdots, L_r]) \, {}^*g\\
& =[X' \xleftarrow {h' \circ \widetilde {g''}} V' \times_V V'  \xrightarrow {f' \circ h' \circ \widetilde {g''}} Y'; (\widetilde {g''})^*(g'')^*L_1, \cdots, (\widetilde {g''})^*(g'')^*L_r]).
\end{align*}
Here we consider the following fiber square:
$$\CD
V' \times _V V' @> {\widetilde {g''}} >> V'\\
@V {\widetilde {g''}} VV @VV {g''} V\\
V'  @>> {g''} > V. \endCD
$$
Thus in general we have that $\frak f \circ g^* \not = \bigl ((g')^* \circ (\, {}^*g) \bigr )\circ \frak f$ in the above diagram (\ref{dia}).
\end{rem}

For the sake of convenience, we list the properties for $\Cal Z^*(X, {\bf Y})$ with $Y$ fixed and $\Cal Z^*({\bf X}, Y)$ with $X$ fixed. Note that to emphasize the target $Y$ or the source $X$ being fixed we denote them in bold, $\bf Y$ and $\bf X$, respectively.

\begin{pro}
\begin{enumerate}
\item 
$\Cal Z^*(X, {\bf Y})$:
\begin{enumerate}
\item (Proper pushforward is covariantly functorial): For two proper maps $f_1:X \to X', f_2:X' \to X''$, 
$(f_1)_*:\Cal Z^i(X,{\bf Y}) \to \Cal   Z^i(X',{\bf Y})$, we have $(f_2)_*:\Cal Z^i(X',{\bf Y}) \to \Cal   Z^i(X'',{\bf Y})$ and 
$$(f_2 \circ f_1)_* = (f_2)_* \circ (f_1)_*.$$

\item (Smooth pullback is contravariantly functorial): For two smooth maps $f_1:X \to X', f_2:X' \to X''$, we have $(f_1)^*:\Cal   Z^{i+\op{dim}f_2}(X',{\bf Y}) \to \Cal   Z^{i+\op{dim}f_2 + \op{dim}f_1}(X,{\bf Y})$, $(f_2)^*:\Cal   Z^i(X'',{\bf Y}) \to \Cal   Z^{i+\op{dim}f_2}(X',{\bf Y})$ and 
$$(f_2 \circ f_1)^* = (f_1)^* \circ (f_2)^*.$$

\item For a line bundle $L$ over $X$ we have the Chern operator:
$$c_1(L) \bullet :\Cal   Z^i(X,{\bf Y}) \to \Cal   Z^{i+1}(X,{\bf Y}).$$
Moreover, if $L$ and $L'$ are isomorphic, then $c_1(L) \bullet =c_1(L') \bullet.$

\item (Proper pushforward and smooth pullback commute) 
 $$\CD
\widetilde X @> {\widetilde f}>> X''\\
@V {\widetilde g} VV @VV {g} V\\
X' @>> {f} > X \endCD
$$ 
with $f$ proper and $g$ smooth, we have
$$g^*f_* = \widetilde f_* \widetilde g^*,$$
i.e., the following diagram commutes:
$$\CD
\Cal   Z^i(X',{\bf Y})@> {f_*}>> \Cal   Z^i(X,{\bf Y})\\
@V {\widetilde g^*} VV @VV {g^*} V\\
\Cal Z^{i+\op{dim}g}(\widetilde X,{\bf Y})  @>> {\widetilde f_*} > \Cal  Z^{i+\op{dim}g}(X'',{\bf Y}) . \endCD
$$
(Note that $\op{dim}\widetilde g = \op{dim}g$.)

\item (compatibility with proper pushforward (``projection formula")): For a proper map $f:X \to X'$ and a line bundle $L$ over $X'$, we have that for $\alp \in \Cal Z^i(X,{\bf Y})$
$$ f_*\Bigl(c_1(f^*L) \bullet \alp \Bigr) = c_1(L)\bullet f_*\alp.$$

\item (compatibility with smooth pullback (``pullback formula")): For a smooth map $f:X' \to X$ and a line bundle $L$ over $X$, we have that for $\alp \in \Cal Z^i(X,{\bf Y})$
$$ f^*\Bigl(c_1(L) \bullet \alp \Bigr) = c_1(f^*L) \bullet f^*\alp.$$

\item (commutativity): If $L$ and $L'$ are line bundles over $X$, then we have
$$c_1(L) \bullet c_1(L') \bullet = c_1(L') \bullet c_1(L) \bullet :\Cal Z^i(X,{\bf Y}) \to \Cal Z^{i+2}(X,{\bf Y}),$$

\end{enumerate}
\item 
$\Cal Z^*({\bf X}, Y)$:
\begin{enumerate}
\item (Smooth pushforward is covariantly functorial): For two smooth maps $g_1:Y \to Y', g_2:Y' \to Y''$, we have 

${}_*(g_1):\Cal   Z^i({\bf X},Y) \to \Cal   Z^{i+\op{dim}g_1}({\bf X},Y')$, 

${}_*(g_2):\Cal Z^{i+\op{dim}g_1} ({\bf X},Y') \to \Cal   Z^{i+\op{dim}g_1 + \op{dim}g_2}({\bf X},Y'')$ and 
$${}_*(g_2 \circ g_1) = {}_*(g_1) \circ {}_*(g_2).$$
\item (Proper pullback is contravariantly functorial): For two proper maps $g_1:Y \to Y', g_2:Y' \to Y''$ and ${}^*(g_1):\Cal   Z^i({\bf X},Y') \to \Cal   Z^i({\bf X},Y)$ and ${}^*(g_2):\Cal   Z^i({\bf X},Y'') \to \Cal   Z^i({\bf X},Y')$
$${}^*(g_2 \circ g_1) = {}^*(g_2) \circ {}^*(g_1).$$
\item For a line bundle $M$ over $Y$ we have the Chern operator:
$$\bullet c_1(M):\Cal   Z^i({\bf X}, Y) \to \Cal   Z^{i+1}({\bf X}, Y).$$
Moreover, if $M$ and $M'$ are isomorphic, then $\bullet c_1(M)=\bullet c_1(L').$

\item (Smooth pushforward and proper pullback commute)
 $$\CD
\widetilde Y @> {\widetilde f}>> Y''\\
@V {\widetilde g} VV @VV {g} V\\
Y' @>> {f} > Y \endCD
$$ 
with $f$ proper and $g$ smooth, we have
$${}^* f \, {}_* g= {}_* \widetilde g \, {}^* \widetilde f,$$
i.e., the following diagram commutes:
$$\CD
\Cal  Z^i({\bf X},Y'')@> {{}_*g}>> \Cal   Z^{i+\op{dim}}({\bf X},Y) \\
@V {{}^* \widetilde f} VV @VV {{}^*f} V\\
\Cal Z^i({\bf X},\widetilde Y)  @>> {{}_*\widetilde g} > \Cal  Z^{i+\op{dim}g}({\bf X},Y') . \endCD
$$
(Note that $\op{dim}\widetilde g = \op{dim}g$.)

\item (compatibility with smooth pushforward (``projection formula")): For a smooth map $g:Y \to Y'$ and a line bundle $M$ over $Y'$ we have that for $\alp \in \Cal Z^i({\bf X},Y)$
$$\Bigl(\alp \bullet c_1(g^*M) \Bigr) \, {}_*g = \alp \, {}_*g \bullet c_1(M).$$

\item (compatibility with proper pullback (``pullback formula")): For a proper map $g:Y' \to Y$ and a line bundle $M$ over $Y$ we have that for $\alp \in \Cal Z^i({\bf X},Y)$
$$\Bigl(\alp \bullet c_1(M) \Bigr) \, {}^*g= \alp \, {}^*g \bullet c_1(g^*M).$$

\item (commutativity): If $M$ and $M'$ are line bundles over $Y$, then we have
$$\bullet c_1(M) \bullet  c_1(M')= \bullet c_1(M') \bullet c_1(M) :\Cal Z^i({\bf X},Y) \to \Cal Z^{i+2}({\bf X},Y),$$
\end{enumerate}
\end{enumerate}
\end{pro}

\vspace{0.5cm}

{\bf Acknowledgements}: The author would like to thank the anonymous referee for his/her careful reading of the manuscript and valuable comments and constructive suggestions. This work was supported by JSPS KAKENHI Grant Numbers JP16H03936 and JP19K03468. \\


\end{document}